\documentclass[12pt]{amsart}

\usepackage{breakurl}
\usepackage{graphicx}
\usepackage{amsmath}
\usepackage{amscd}
\usepackage{amsfonts}
\usepackage{amssymb}
\usepackage[dvipsnames]{xcolor}
\usepackage{fullpage}
\usepackage{mathtools}
\usepackage[pagebackref,hypertexnames=false, colorlinks, citecolor=red,linkcolor=blue, urlcolor=red, bookmarksopen]{hyperref}
\usepackage{todonotes}
\usepackage{tikz-cd}
\usepackage{enumitem}
\usepackage{amsthm}

\usepackage[depth=3]{bookmark}

\numberwithin{equation}{section}

\newtheorem{theorem}{Theorem}[section]
\newtheorem{lemma}[theorem]{Lemma}
\newtheorem{proposition}[theorem]{Proposition}

\newtheorem{corollary}[theorem]{Corollary}

\theoremstyle{definition}
\newtheorem{example}[theorem]{Example}
\newtheorem{remark}[theorem]{Remark}

\newcommand{\be}{\begin{equation}}
\newcommand{\ee}{\end{equation}}
\newcommand{\bes}{\begin{equation*}}
\newcommand{\ees}{\end{equation*}}

\newcommand{\cA}{\mathcal{A}}
\newcommand{\cB}{\mathcal{B}}

\newcommand{\cE}{\mathcal{E}}
\newcommand{\cF}{\mathcal{F}}
\newcommand{\cG}{\mathcal{G}}
\newcommand{\cH}{\mathcal{H}}
\newcommand{\cI}{\mathcal{I}}
\newcommand{\cJ}{\mathcal{J}}
\newcommand{\cK}{\mathcal{K}}
\newcommand{\cL}{\mathcal{L}}

\newcommand{\cO}{\mathcal{O}}

\newcommand{\cR}{\mathcal{R}}
\newcommand{\cS}{\mathcal{S}}

\newcommand{\cU}{\mathcal{U}}
\newcommand{\cV}{\mathcal{V}}

\newcommand{\lel}{\left\langle}
\newcommand{\rir}{\right\rangle}

\newcommand{\bB}{\mathbb{B}}
\newcommand{\bC}{\mathbb{C}}
\newcommand{\bD}{\mathbb{D}}

\newcommand{\bF}{\mathbb{F}}
\newcommand{\bM}{\mathbb{M}}
\newcommand{\bN}{\mathbb{N}}

\newcommand{\ol}{\overline}

\newcommand{\im}{\operatorname{Im}}

\newcommand{\Rep}{\operatorname{Rep}}

\newcommand{\alg}{\operatorname{alg}}
\newcommand{\Aut}{\operatorname{Aut}}
\newcommand{\diag}{\operatorname{diag}}

\newcommand{\Mult}{\operatorname{Mult}}
\newcommand{\mlt}{\operatorname{Mult}}

\newcommand{\spn}{\operatorname{span}}

\newcommand{\fB}{{\mathfrak{B}}}
\newcommand{\fC}{{\mathfrak{C}}}
\newcommand{\fD}{{\mathfrak{D}}}

\newcommand{\fL}{{\mathfrak{L}}}

\newcommand{\fV}{{\mathfrak{V}}}
\newcommand{\fX}{{\mathfrak{X}}}

\newcommand{\foral}{\text{ for all }}

\begin{document}

\title{On the Classification of Function Algebras on Subvarieties of Noncommutative Operator Balls}

\author{Jeet Sampat}
\address{Department of Mathematics\\
Technion --- Israel Institute of Technology\\
Haifa, Israel}
\email{sampatjeet@campus.technion.ac.il}
\author{Orr Moshe Shalit}
\address{Department of Mathematics\\
Technion --- Israel Institute of Technology\\
Haifa, Israel}
\email{oshalit@technion.ac.il}


\thanks{The second author was partially supported by Israel Science Foundation Grant no. 431/20.}

\begin{abstract}
We study algebras of bounded noncommutative (nc) functions on unit balls of operator spaces (nc operator balls) and on their subvarieties. 
Considering the example of the nc unit polydisk we show that these algebras, while having a natural operator algebra structure, might not be the multiplier algebra of any reasonable nc reproducing kernel Hilbert space (RKHS). 
After examining additional subtleties of the nc RKHS approach, we turn to study the structure and representation theory of these algebras using function theoretic and operator algebraic tools. 
We show that the underlying nc variety is a complete invariant for the algebra of uniformly continuous nc functions on a homogeneous subvariety, in the sense that two such algebras are completely isometrically isomorphic if and only if the subvarieties are nc biholomorphic. 
We obtain extension and rigidity results for nc maps between subvarieties of nc operator balls corresponding to injective spaces that imply that a biholomorphism between homogeneous varieties extends to a biholomorphism between the ambient balls, which can be modified to a linear isomorphism. 
Thus, the algebra of uniformly continuous nc functions on nc operator balls, and even its restriction to certain subvarieties, completely determine the operator space up to completely isometric isomorphism.
\end{abstract}

\maketitle

\section{\textbf{Introduction}} \label{section_introduction}

\subsection{Overview} \label{subsection_overview}

Algebras of bounded noncommutative (nc) functions on nc domains and their subvarieties constitute an interesting class of objects to study for several reasons. 
First, ``because they are there" --- they form a rich class of operator algebras that are amenable to analysis in concrete terms that we can investigate, and thus broaden our acquaintance with operator algebras. 
Second, it turns out that this class of algebras contains many operator algebras that have arisen independently of nc function theory, for example in \cite{SSS18,SSS20} the class of algebras of bounded nc functions on subvarieties of the nc unit row ball were shown to include tensor algebras of subproduct systems \cite{KakSha19,ShaSol09} and also (commutative) multiplier algebras of complete Pick spaces \cite{DRS11, DRS15, Har12}. 
Representing these operator algebras as algebras of nc functions provides us with tractable invariants, enlightening insights, and hence also with new results; e.g., applications to isomorphisms of subproduct system algebras in \cite{SSS18} or to the problem of quasi-inner automorphisms of $\fL_d$ in \cite{SSS20}.
Third, studying these algebras, and in particular classifying them, has driven the discovery of purely function theoretic and geometric theorems; for example: maximum modulus principle and Nullstellens\"{a}tze \cite{SSS18}, nc Schwarz lemma and spectral Cartan uniqueness theorem \cite{SSS20}, fixed point and iteration theoretic results \cite{BS+,Sham18}, and the clarification of the notion of uniform continuity \cite{HRS22}. 
In this paper, among other things, we find a fourth reason why these algebras are interesting: they provide a concrete representation of the universal unital operator algebra generated by an operator space, and this leads us to some new striking observations of a general nature. 

While operator algebras of noncommutative functions have been studied in many versions for a long time (see \cite{AriLat12,AriPop00,DavPitts98b,KakSha19,MS98,MS04,MS05,Pop96,Pop06,Pop11} for a sample), it seems that most works have dealt with functions on ball (or ``polyball") type domains and their subvarieties. 
In particular, the research that led to this paper was motivated by previous work on algebras of bounded nc functions on subvarieties of the nc unit row ball \cite{SSS18,SSS20}, and began as an attempt to understand the extent to which those results really depend on the underlying ambient nc domain i.e., the nc unit row ball\footnote{In \cite{SSS18,SSS20} $\fB_d$ was simply referred to as the {\em nc unit ball} since there was no place for confusion.}
\begin{equation}\label{eq:ball}
\fB_d = \left \{X \in \bM^d : \left \|\sum X_j X_j^* \right \| < 1\right \}. 
\end{equation}
What if the underlying domain is some other nc unit ball? 
We discovered that, on the one hand, $\fB_d$ is indeed quite special, and some results or techniques that worked in that special case are out of the question in general. 
On the other hand, we found that some important results hold in a vast generality, and in particular nc varieties are a complete invariant for algebras of nc functions, in an appropriate sense --- this required modifying many of the techniques and resulted in several results of independent interest. 
We also found that classifying nc function algebras on subvarieties embedded in different nc domains opened up new questions, to which we found unexpected answers suggesting new phenomena.

\subsection{Definitions and notation} \label{subsection_definitions_notation}

We use the theory of noncommutative (nc) functions as developed by Agler and McCarthy \cite{AM15a,AM15b,AM16}, Helton, Klep and McCullough \cite{HKM11a,HKM11b,HKM12}, and Kaliuzhnyi-Verbovetskyi and Vinnikov \cite{KVV14}. 
We also require the theory of nc reproducing kernel Hilbert spaces (RKHS) as developed by Ball, Marx and Vinnikov \cite{BMV16,BMV18}. 

\subsubsection{\textbf{NC sets}}
For $n, d \in \bN$, let $M_{n}^d$ denote the collection of $d$-tuples $X = (X_1, \dots, X_d)$ of $n \times n$ complex matrices and let $M_n = M^1_n$. 
The {\em nc universe} $\bM^d$ is defined to be the graded union
\[
\bM^d = \sqcup_{n=1}^\infty M_n^d.
\]

Recall that a {\em nc set} is a subset $\Omega \subseteq \bM^d$ that is closed under direct sums. 
A nc set $\Omega$ is said to be a {\em nc domain} if it is open at every level i.e., $\Omega(n) := \Omega \cap M_n^d$ for every $n \in \bN$. 
We shall write $\ol{\Omega}$ for the level-wise closure $\sqcup_n \ol{\Omega(n)}$.

\subsubsection{\textbf{NC functions}}
A function $f$ from a nc set $\Omega \subseteq \bM^d$ to $\mathbb{M}^l$ is said to be a {\em nc function} if: \begin{enumerate}

\item  $f$ is {\em graded}: $X \in \Omega(n) \Rightarrow f(X) \in M^l_n$;

\item $f$ {\em respects direct sums}: $X \in \Omega(n)$, $Y \in \Omega(m) \Rightarrow f(X \oplus Y) = f(X) \oplus f(Y)$; and

\item $f$ {\em respects similarities}: $S \in GL_n$, $X, S^{-1} X S \in \Omega(n) \Rightarrow f(S^{-1} X S) = S^{-1} f(X) S$.

\end{enumerate} Let $\bF_d^+$ denote the free unital semigroup generated by $d$ generators $1, \ldots, d$. 
An element $\alpha \in \bF_d^+$ is a free word $\alpha = \alpha_1 \cdots \alpha_k$ with length $|\alpha| := k \in \bN$ and $\alpha_j \in \{1, \ldots, d\}$ (for $1 \leq j \leq k$). 
We write $Z^\alpha = Z_{\alpha_1}\cdots Z_{\alpha_k}$ for the multi-variable $Z = (Z_1, \ldots, Z_d)$, and likewise for tuples of operators.
Thus, a {\em free nc polynomial} (or simply, {\em polynomial}) is a finite sum 
\[
p(Z) = \sum_{\alpha \in \bF_d^+} c_{\alpha} Z^{\alpha}
\]
that can be evaluated on every $X \in \bM^d$ in the obvious way: $p(X) = \sum_{\alpha} c_{\alpha} X^{\alpha}$. 
Free nc polynomials are evidently nc functions. We denote $ \bC \lel Z \rir := \bC\langle Z_1, \ldots, Z_d \rangle$ to be the algebra of all such free nc polynomials in $d$ (free) noncommuting variables.

More generally, given any two Hilbert spaces $\cR$ and $\cS$, an \emph{operator-valued free nc polynomial} is a map $Q : \bM^d \rightarrow \cL(\cR,\cS)_{nc}$ given by the formal expression
\[
Q(Z) = \sum_{\alpha \in \bF_d^+} Q_{\alpha} Z^{\alpha},
\]
where $Q_{\alpha} \in \cL(\cR,\cS)$ are all $0$ except for finitely many $\alpha \in \bF_d^+$. Here, we use $\cL(\cR,\cS)$ to denote the space of bounded operators between $\cR$ and $\cS$, and as in \cite{BMV16,BMV18} we identify
\begin{equation} \label{equation_ambient_operator_space}
\cL(\cR,\cS)_{nc} := \sqcup_{n = 1}^\infty \cL(\cR^n,\cS^n) = \sqcup_{n = 1}^\infty M_n(\cL(\cR,\cS)).
\end{equation}
$Q$ is interpreted functionally as $Q(X) = \sum_\alpha Q_{\alpha} \otimes X^{\alpha}, \, \forall X \in \bM^d$.

\subsubsection{\textbf{NC operator balls}}
Two nc domains that will play a leading role in this paper are the {\em nc unit row ball} \eqref{eq:ball}, and the {\em nc unit polydisk} which is defined as $\fD_d = \sqcup_{n=1}^\infty \fD_d(n)$, where 
\[
\fD_d(n) = \Big \{X \in M_n^d : \|X\|_{\infty} := \max_{1\leq j\leq d} \|X_j\| <1 \Big \}.
\]

More generally, every $d$-dimensional operator space $\cE$ defines a nc domain that we can think of as the nc unit ball of $\cE$. 
Given an operator space $\cE \subseteq \cL(\cR,\cS)$, and given a basis $\{Q_1, \ldots, Q_d\}$ for $\cE$, we consider the linear operator-valued polynomial $Q(Z) = \sum_{j=1}^d Q_j Z_j$. 
The corresponding \emph{nc operator ball} $\bD_Q$ is defined as 
\begin{equation} \label{eq:operator_ball}
\bD_Q = \{ X \in \bM^d : \| Q(X) \| < 1 \}.
\end{equation} 
For any $n \in \bN$, $\bD_Q(n)$ can be naturally identified with the open unit ball of $M_n(\cE)$ via 
\begin{equation} \label{eq:identify_DQ}
\bD_Q(n) \ni X \longleftrightarrow Q(X) \in M_n(\cE) \subset \cL(\cR^n,\cS^n).
\end{equation}
In fact, $Q$ gives rise to a linear isomorphism between $M_n^d$ and $M_n(\cE)$ for every $n \in \bN$. This shows that $\bD_Q$ is nonempty, open and bounded (as it is the preimage of the unit ball under a linear isomorphism). 
Note that every choice of basis $\{Q_1, \ldots, Q_d\}$ for $\cE$ gives a different domain $\bD_Q$, but all these nc domains are in some sense equivalent. 

\begin{example} 
Both $\fB_d$ and $\fD_d$ are nc operator balls: 
\begin{enumerate}	
\item If $Q(Z) = [ Z_1 \dots Z_d ]$ then $\bD_Q = \fB_d$, corresponding to the row operator Hilbert space. 
\item If $Q(Z) = \diag(Z_1, \dots, Z_d)$ then $\bD_Q = \fD_d$, corresponding to $\cE = \min(\ell^\infty(\bC^d))$. 
\end{enumerate}
\end{example}

\begin{remark} 
One might wish to generalize further and to consider also domains $\bD_Q$ given by a polynomial $Q$ of higher degree.
However, we shall require our domain to be \emph{homogeneous} i.e., $\lambda \bD_Q \subseteq \bD_Q$ for $\lambda \in \bD$, and it is not clear how to guarantee that without requiring that $Q$ be homogeneous i.e., there exists $q \in \bN$ such that 
\[
Q(\lambda Z) = \lambda^q Q(Z), \, \forall \lambda \in \bC.
\]
In this case, if $q\geq 2$ then the corresponding \emph{nc homogeneous polyhedron} $\bD_Q$ defined by \eqref{eq:operator_ball} will always be unbounded since it contains all \emph{jointly $q$-nilpotent $d$-tuples}: 
\[
X = (X_1, \dots, X_d) \text{ s.t. } X^\alpha = 0, \, \forall \alpha \in \bF_d^+ \text{ with } |\alpha| = q.
\] 
In particular, the coordinate functions $Z_1, \dots, Z_d$ will be unbounded on such a polyhedron.
\end{remark}

\subsubsection{\textbf{NC function algebras}}
It is well-known that bounded nc functions are automatically continuous and holomorphic in all \emph{nc topologies} of interest (see \cite[Chapter 7]{KVV14}).
Let $H^\infty(\bD_Q)$ denote the algebra of bounded nc functions on $\bD_Q$, equipped with the sup-norm:
\[
\|f\|_\infty := \sup_{X \in \bD_Q} \|f(X)\|, \, \forall f \in H^\infty(\bD_Q).
\] 
This algebra contains all free nc polynomials, and also all nc power-series with coefficients that converge sufficiently fast. 
In fact, since $\bD_Q$ is \emph{uniformly-open} (see Section \ref{subsection_bounded_nc_maps_on_D_Q}), \cite[Theorem 7.21]{KVV14} shows us that every $f \in H^\infty(\bD_Q)$ has a global nc power-series representation:
\[
f(X) = \sum_{\alpha \in \bF_d^+} c_\alpha X^\alpha, \, \forall X \in \bD_Q.
\]
We let $A(\bD_Q)$ denote the closure of the free nc polynomials in $H^\infty(\bD_Q)$.

Matrix valued $H^\infty(\bD_Q)$ functions can be naturally endowed with a family of matrix norms: 
\[
\left\| \left [ F_{j.k}(Z) \right ] \right\|_{m \times n} := \sup_{X \in \bD_Q} \left \| \left[F_{j,k}(X)\right] \right \|
\] 
for all $\left[F_{j,k}\right] \in M_{m \times n}\left(H^\infty(\bD_Q)\right) \textrm{ and } m,n \in \bN$, which satisfy the Blecher-Ruan-Sinclair axioms. Thus, $H^\infty(\bD_Q)$ and $A(\bD_Q)$ are operator algebras (see \cite[Theorem 2.3.2]{BLM04}). 

\subsubsection{\textbf{NC varieties and their algebras}}
A {\em nc subvariety} of $\bD_Q$ is a nc set of the form 
\[
\fV = \{X \in \bD_Q : f(X) = 0, \, \forall f \in S\},
\]
where $S$ is some subset of $H^\infty(\bD_Q)$. 
We define $H^\infty(\fV)$ to be the algebra of bounded nc functions on $\fV$, and $A(\fV)$ to be the norm-closure of polynomials in $H^\infty(\fV)$. 
The algebras $H^\infty(\fV)$ and $A(\fV)$ are also operator algebras, clearly.

\subsection{Main results} \label{subsection_main_results}

In \cite{SSS18} the operator algebras $H^\infty(\fV)$ for general subvarieties $\fV \subseteq \fB_d$, and the algebras $A(\fV)$ for homogeneous subvarieties $\fV \subseteq \fB_d$, were classified up to completely isometric isomorphism. 
The main results can be summarized by saying, roughly, that two algebras are completely isometrically isomorphic if and only if the varieties they live on are biholomorphic; further, it was shown that if the varieties are biholomorphic then one must, in fact, be the image of the other under an automorphism of the nc unit ball (this requires also later results from \cite{BS+,Sham18}) and in the homogeneous case under a unitary map. 
Classification up to bounded isomorphism was taken up in \cite{SSS20}. 

These works relied on the special properties of the nc unit row ball in several ways. 
The fact that $H^\infty(\fB_d)$ happens to be the nc multiplier algebra of a nc reproducing kernel Hilbert space allowed to represent it naturally, and work with a weak-operator topology and use spatial techniques such as dilation and compression. 
The fact that this nc RKHS has the complete Pick property further helped extend functions and maps from subvarieties to $\fB_d$. 
The fact that $H^\infty(\fB_d)$ can be identified with the noncommutative analytic Toeplitz algebra $\fL_d$ allowed us to use the worked out functional calculus and representation theory due to Arias and Popescu \cite{AriPop00}, Davidson and Pitts \cite{DavPitts98b}, Muhly and Solel \cite{MS11}, and Popescu \cite{Pop95}. 
Using this representation theory nc varieties were identified as representation spaces, and then further geometric properties specific to the nc unit row ball were used to analyze mappings between these balls. 
The classification up to bounded isomorphism was made possible by a detailed understanding of the similarity envelope $\widetilde{\fB}_d$ of the nc unit row ball. 

In this paper we set out to study similar questions in the vast generality of nc operator balls, where most of the above mentioned tools are missing. 
We now outline each remaining section of this paper for the reader's convenience.

\subsubsection{\textbf{Section \ref{section_nc_function_theory_in_polydisk}}}

Our work begins with the crucial observation that it is not possible to identify the algebras $H^\infty(\fV)$ as multiplier algebras of nc RKHSs in general. 
To be precise, in Theorem \ref{thm:mult_poly} we show that $H^\infty(\fD_d)$ is not the multiplier algebra of any nc RKHS for which the monomials form an orthogonal basis.

One might think that maybe there is a natural quantization of the Hardy space of the commutative unit polydisk for which the multiplier algebra is not $H^\infty(\fD_d)$ but may still be an algebra worth investigating. However, in Theorem \ref{thm_multiplier_algebra_nc_hardy_space} we show that one such natural quantization has no non-trivial multipliers.

We also compute the C*-envelope of $A(\fD_d)$ in Theorem \ref{thm:C*-envelope_A(D_d)}, and show that it is equal to the full group C*-algebra $C^*(\bF_d)$ of the free group $\bF_d$. This is in striking contrast to what happens in the case of $\fB_d$, where $C^*_e(A(\fB_d)) = \cO_d$ --- the Cuntz algebra \cite{MS98}.

\subsubsection{\textbf{Section \ref{section_nc_operator_balls}}}

With some of our hopes extinguished, we continue in the remainder of the paper to see what we can figure out, after all. In Theorem \ref{theorem_bdy_value_principle}, we prove a boundary value principle that replaces the maximum modulus principle in the nc unit row ball. In the remainder of the section, we study the structure of $A(\bD_Q)$ and $H^\infty(\bD_Q)$ as topological algebras.

\subsubsection{\textbf{Section \ref{section_cAV_as_quotients}}}

Let us use $(\cA,\tau)$ to denote either $A(\bD_Q)$ with the norm topology, or $H^\infty(\bD_Q)$ with the topology of bounded-pointwise convergence. Likewise, we let $\cA(\fV)$ be either $A(\fV)$ or $H^\infty(\fV)$. In Theorem \ref{thm_A(V)_is_quotient_of_A(D_d)}, we show that the restriction map $F \mapsto F\big|_{\fV}$ is a complete surjection that induces a completely isometric isomorphism between $\cA/I(\fV)$ and $\cA(\fV)$. When $\cA = A(\bD_Q)$ this is more difficult and we need to assume that $\fV$ is homogeneous.

For homogeneous ideals, we also have a Nullstellensatz (Theorem \ref{thm_nullstellensatz_for_homogenous_poyhedra}) that clarifies the nature of the quotient; see Corollary \ref{corollary_isomorphism_subalgebras_of_A}. We close the section with Proposition \ref{proposition_uniform_continuity_of_A(V)_functions} and show that $A(\fV)$ is precisely the algebra of nc functions that are uniformly continuous on $\fV$. 

\subsubsection{\textbf{Section \ref{section_representations}}}

We study the finite dimensional completely contractive (cc) representations of the algebras $\cA(\fV)$. 
In general, what we can say is limited: there exists a continuous surjection $\pi: \Rep^{cc}_n(\cA) \to \ol{\bD_Q}$ given by evaluating the representation on the coordinate functions:
\[
\pi(\Phi) = (\Phi(Z_1), \ldots, \Phi(Z_d)), \, \forall \Phi \in \Rep^{cc}_n(\cA).
\]

In the other direction, every $X \in \fV$ gives rise to an evaluation map $\Phi_X \in \Rep^{cc}_n(\cA(\fV))$ given by $\Phi_X(f) = f(X)$, for every $f \in \cA(\fV)$. For the algebras $A(\fV)$ when $\fV$ is a homogeneous variety, these maps provide us with the identification $\Rep^{cc}_n(A(\fV)) \cong \ol{\fV}$.

\subsubsection{\textbf{Section \ref{section_classification_of_A(V)}}}

The above concrete description of the finite dimensional cc representations allows us to obtain a basic classification result (Theorem \ref{thm:iso_cont}) for homogeneous subvarieties of nc operator balls. We then focus on certain special $\bD_Q$ which correspond to injective operator spaces and that we refer to as \emph{injective nc operator balls}. 
The remarkable work of Ball, Marx and Vinnikov \cite{BMV18} gives us a powerful extension theorem (restated in Theorem \ref{theorem_BMV_interpolation_full_sets}), that we refine in Theorem \ref{thm:extend_to_ball} to the statement that a nc map on a relatively full subset of a nc operator ball $\bD_{Q_1}$ mapping into an injective nc operator ball $\bD_{Q_2}$ can be extended to a nc map defined on the whole of $\bD_{Q_1}$ that maps it into $\bD_{Q_2}$. This ``nonlinear injectivity" phenomenon takes the place of the complete Pick property used for $\fB_d$.

In Theorem \ref{thm:iso_cont_1-injective}, we fully solve the isomorphism problem for $A(\fV)$ when $\fV$ is a homogeneous subvariety of some injective nc operator ball $\bD_Q$ by showing that if $\fV_i \subseteq \bD_{Q_i}$ are homogeneous subvarieties (for $i=1,2$) then there exists a completely isometric isomorphism $\varphi: A(\fV_1) \to A(\fV_2)$ if and only if there exists a uniformly continuous nc biholomorphism $F \colon \fV_2 \rightarrow \fV_1$ with a uniformly continuous inverse such that $\varphi(f) = f\circ F, \foral f \in A(\fV_1)$.

We analyze certain rigidity properties of maps between homogeneous varieties, and show that if $\fV_1$, $\fV_2$ above are also {\em matrix-spanning} then they are biholomorphic in a stronger sense. To be precise, in Theorem \ref{theorem_cartan_uniqueness_for_subvariety} we show that there exists a nc biholomorphism $\widetilde{F} : \bD_{Q_2} \to \bD_{Q_1}$ such that $\widetilde{F}\vert_{\fV_2} = F$, where $F$ is as above. We provide a concrete description of matrix-spanning subvarieties in Theorem \ref{theorem_mat-spanning_subvarieties}, and discuss (with examples) the situation for nc operator balls that are not injective in Section \ref{subsection_examples}.

In Theorem \ref{theorem_nc_biholomorphism_that_is_linear}, we provide a general result that shows that the existence of a nc biholomorphism between $\bD_{Q_1}$ and $\bD_{Q_2}$ that maps homogeneous subvarieties $\fV_1$ onto $\fV_2$ implies the existence of a {\em linear isomorphism} between the balls that maps $\fV_1$ onto $\fV_2$.
In order to prove this theorem, we require a new approach based on Braun, Kaup and Upmeier's rigidity results \cite{BKU78}. Our ultimate classification result is Theorem \ref{theorem_isomorphism_A(D_Q)_with_nc_biholomorphism_of_domains}, which says that for matrix spanning varieties, the following statements are equivalent: 
\begin{enumerate}

\item There exists a completely isometric isomorphism $\alpha : A(\fV_1) \rightarrow A(\fV_2)$.

\item There exists a nc biholomorphism $G : \bD_{Q_1} \to \bD_{Q_2}$ such that $G(\fV_1) = \fV_2$.

\item There exists a linear isomorphism $L : \bD_{Q_1} \rightarrow \bD_{Q_2}$ such that $L({\fV_1}) = \fV_2$.

\end{enumerate}

\subsubsection{\textbf{Section \ref{section_concluding}}}

We close this paper with some concluding remarks that put our concrete results into an abstract-minded perspective. 
After all, with every operator space we have associated a natural function algebra. 
What is this algebra? 
By reinterpreting some of our results we show that it is the \emph{universal free unital algebra} over the dual $\cE^*$, leading to some insights.

\section{\textbf{Bounded Functions Vs. Multipliers: The NC Unit Polydisk as a Test Case}} \label{section_nc_function_theory_in_polydisk}

In commutative function theory, RKHSs and their multiplier algebras play a prominent role \cite{AMBook}. 
Multiplier algebras are closely related to bounded analytic functions. 
For instance, the Hardy space $H^\infty(D)$ for a bounded domain $D$ always occurs as the multiplier algebra of some RKHS. 
Conversely, for most natural RKHSs on a bounded domain $D \subset \bC^d$, the multiplier algebra is contained in $H^\infty(D)$, and in many cases equal to it. 
For the Drury-Arveson space $H^2_d$ --- a RKHS on the commutative unit ball $\bB_d$ --- the multiplier algebra is famously {\em not} equal to $H^\infty(\bB_d)$ (see \cite{Arv98,DavPitts98a}). However, in \cite[Section 11]{SSS18} it was shown that $\Mult(H^2_d)$ can be naturally identified with $H^\infty(\fC\fB_d)$, namely, the algebra of bounded nc functions on the subvariety $\fC\fB_d \subset \fB_d$ that consists of all commuting row contractions. 
Furthermore, when $\fV \subseteq \fB_d$, then $H^\infty(\fV)$ is the algebra of nc (left) multipliers on a certain nc RKHS on $\fV$, and in particular $H^\infty(\fB_d)$ is the multiplier algebra of the so-called nc Drury-Arveson space $\cH^2_d$. 
This was used extensively in the study of these algebras in \cite{SSS18,SSS20}.

As noted in Section \ref{subsection_definitions_notation}, the algebras $H^\infty(\bD_Q)$ are always abstract operator algebras. 
This raises the question of whether these operator algebras have an intuitive faithful representation on a Hilbert space, and in particular whether $H^\infty(\bD_Q)$ is the multiplier algebra of a nc RKHS on $\bD_Q$.
In this section, we show that in the noncommutative world the connection between bounded analytic functions and multiplier algebras is broken.

First, we show that multipliers need not be bounded at all, but they will always be \emph{spectrally bounded}. 
Then, we turn to exhibit several obstacles to studying function theory on the nc unit polydisk via nc RKHSs.
The theory of nc RKHSs was developed by Ball, Marx and Vinnikov in \cite{BMV16}. A nc RKHS is a Hilbert space $\cH$ of nc functions on a nc set $\Omega$ such that for every $X \in \Omega(n)$ and every $v,y \in \bC^n$ there is a \emph{kernel function} $K_{X,v,y} \in \cH$ such that 
\[
\langle h(X)v, y \rangle_{\bC^n} = \langle h , K_{X,v,y} \rangle_\cH, \, \forall h \in \cH.
\]
The algebra $\Mult(\cH)$ of nc (left) multipliers consists of all nc functions $f: \Omega \to \bM^1$ such that the point-wise product $f \cdot h$ is in $\cH$ for all $h\in \cH$.

\subsection{Boundedness of multipliers} \label{subsection_on_boundedness_of_multipliers}

By a basic result in the theory of RKHS, every multiplier $f$ on a (classical) RKHS is bounded. In fact, \cite[Equation 2.33]{AMBook} shows that
\[
\|f\|_\infty \leq \|M_f\|.
\] 
In the noncommutative setting, the following example shows that this no longer holds. 

\begin{example}
Let $\cH_d^2$ be the nc Drury-Arveson space on the nc unit row ball $\fB_d$. 
It follows from \cite[Corollary 3.7]{SSS18} that $\mlt(\cH_d^2) = H^\infty(\fB_d)$ completely isometrically, and therefore $\|f\|_\infty = \|M_f\|$. 
However, as in \cite{SSS20}, every nc function on $\fB_d$ can be extended uniquely as a nc function to the \emph{similarity envelope} $\widetilde{\fB}_d$ of $\fB_d$: \[
\widetilde{\fB_d} := \{ X \in \bM^d \, : \, S^{-1} \cdot X \cdot S := (S^{-1}X_1S,\dots,S^{-1}X_dS) \in \fB_d, \, S \text{ -- invertible} \}.
\] Therefore $\cH_d^2$ becomes also a nc RKHS on $\widetilde{\fB}_d$, and $H^\infty(\fB_d)$ determines its multiplier algebra by way of extension. 
Clearly, the coordinate function $Z_1 : (X_1, \ldots, X_d) \mapsto X_1$ is a bounded multiplier of norm $1$, however as a function on $\widetilde{\fB}_d$ it is unbounded.
\end{example}

\begin{remark}
	
We do not have such an example when the domain $\Omega \subset \bM^d$ of the nc RKHS $\mathcal{H}$ is bounded, nor does it follow immediately from the definitions that multipliers of $\mathcal{H}$ must be bounded nc functions on $\Omega$. We can however, even with no boundedness assumption on $\Omega$, show that multipliers must always be \emph{spectrally bounded}.

\end{remark}

Let $\mathcal{H}$ be a Hilbert space. Recall that the {\em spectral radius} of $T \in \cL(\mathcal{H})$ is defined to be \[
\rho(T) = \sup\{|\lambda| : \lambda \in \sigma(T) \},
\] where $\sigma(T)$ is the spectrum of $T$. Gelfand's spectral radius formula states that \[
\rho(T) = \lim_{n \to \infty} \|T^n\|^{1/n}.
\]

\begin{proposition}
Let $\cH$ be a nc RKHS on a nc set $\Omega$, and suppose that either $(i)$ the constant function $1 $ lies in $ \cH$; or that $(ii)$ for all $n$, $X \in \Omega(n)$ and $y \in \bC^n$, there exists $v \in \bC^n$ such that $K_{X,v,y} 
\neq 0$. 
Then, for all $f \in \mlt(\cH)$ we have
\[
\sup_{X \in \Omega} \rho(f(X)) \leq \|M_f\|.
\]
\end{proposition}
\begin{proof}
For $X \in \Omega_n$ and unit vectors $v,y \in \bC^n$, consider the kernel function $K_{X,v,y}$. 
If we assume that $1 \in \cH$, then
\[
\langle f(X) v, y \rangle = \langle M_f 1 , K_{X,v,y} \rangle, \, \forall f \in \mlt(\cH).
\]
Taking the supremum over unit vectors $v,y \in \bC^n$ we get \[
\|f(X)\| \leq C_X \|M_f\|,
\] where the constant $C_X$ satisfies \[
C_X = \|1\| \sup_{\|v\|=\|y\|=1} \|K_{X,v,y}\| \leq \| 1 \| \sqrt{\|K(X,X)\|}.
\] Applying this inequality to the powers of $f$ we get \[
\|f(X)^n\| \leq C_X \|M_f^n\| \leq C_X \|M_f\|^n, \, \forall n \in \bN.
\] Taking $n^{\text{th}}$ roots on both sides and letting $n \to \infty$ we obtain \[
\rho(f(X)) \leq \|M_f\|.
\] As $X \in \Omega_n$ was arbitrarily chosen, we obtain the claimed inequality.

Alternatively, if $X \in \Omega(n)$, let $\lambda \in \sigma(f(X))$ and $y \in \bC^n$ be such that $f(X)^*y = \ol{\lambda} y$. 
By assumption, there exists $v \in \bC^n$ such that $K_{X,v,y} \neq 0$. 
Then we have 
\[
\|M^*_f K_{X,v,y}\| = \|K_{X,v,f(X)^*y}\| = |{\lambda}| \|K_{X,v,y} \|, 
\]
whence $| \lambda | \leq \|M_f\|$, as required. 
\end{proof}

\subsection{$H^\infty(\fD_d)$ is not a multiplier algebra} \label{subsection_no_Hinfty_multipliers}

The following result shows that $H^\infty(\fD_d)$ can never be the multiplier algebra of any `reasonable' nc RKHS on $\fD_d$.

\begin{theorem}\label{thm:mult_poly}
Fix $d \geq 2$. Let $\cH$ be a nc RKHS on $\fD_d$ that contains the free nc polynomials, and for which the nc monomials form an orthogonal basis. 
Then, $\mlt(\cH) \neq H^\infty(\fD_d)$. 
\end{theorem}
\begin{proof}
Suppose that $\mlt(\cH) = H^\infty(\fD_d)$. 
Consider the coordinate functions $Z_1, \ldots, Z_d$ and form the row operator $[M_{Z_1} \ldots M_{Z_d}] \in M_{1,d}(H^\infty(\fD_d))$. Then, on the one hand, we have
\[
\left\|\left[M_{Z_1} \ldots M_{Z_d}\right]\right\| = \sup_{X \in \fD_d} \left\|\left[X_1 \ldots X_d \right ] \right \| = \sqrt{d} > 1.
\]
The second equality holds since the Cauchy-Schwarz inequality implies an inequality `$\leq \sqrt{d}$':
\begin{align*}
\left\| \sum_{j = 1}^d X_j h_j \right \|^2 &\leq \left(\sum_{j=1}^d  \| X_j h_j \|\right)^2 \\
& \leq \left(\sum_{j=1}^d \|X_j h_j\|^2 \right) \left(\sum_{j=1}^d 1^2 \right) \\
&\leq d, \, \forall X \in \fD_d,
\end{align*}
and the choice of $X = [I_n \dots I_n]$ for some identity matrix $I_n$ implies an inequality `$\geq \sqrt{d}$'.

On the other hand, if $\mlt(\cH) = H^\infty(\fD_d)$ then we have \[
\|M_{Z_j}\| = \sup_{X  \in \fD_d} \|X_j\| = 1, \, \forall 1 \leq j \leq d.
\]
If the monomials are orthogonal, the multipliers $M_{Z_1}, \ldots, M_{Z_d}$ have orthogonal ranges. Thus,
\[
\|[M_{Z_1} \ldots M_{Z_d}]\|^2 = \left\|\sum_{j=1}^d M_{Z_j} M_{Z_j}^*\right\| \leq 1. 
\]
It follows that if $\cH$ is as in the statement of the proposition then $\mlt(\cH) \neq H^\infty(\fD_d)$.
\end{proof}



\subsection{Candidate for a NC Hardy space on $\fD_d$} \label{subsection_candidate_for_nc_Hardy_space}
The above result motivates us to come up with examples of Hilbert spaces of nc functions on the nc unit polydisk, and to understand their multiplier algebras. 
We use the commutative unit polydisk $\bD^d := \bD \times \dots \times \bD$ as inspiration. 
It is well-known that for $\bD^d$, the algebra $H^\infty(\bD^d)$ of bounded holomorphic functions on $\bD^d$ is the multiplier algebra of the classical {\em Hardy space $H^2(\bD^d)$}
\[
H^2(\bD^d) = \left\{ f \sim \sum_{\alpha \in \bN^d} c_\alpha z^\alpha : \sum_{\alpha \in \bN^d} |c_\alpha|^2 < \infty \right\}.
\]

The corresponding Hilbert space of formal nc power-series with square summable coefficients turns out to be the nc Drury-Arveson space $\cH_d^2$, but it is well-known that $\cH_d^2$ {\em symmetrizes} to the commutative Drury-Arveson space on the unit-ball $\bB_d$ instead of $H^2(\bD^d)$. Therefore it is natural to ask if a nc RKHS that symmetrizes to $H^2(\bD^d)$ exists. In this section we find a candidate for such a space, and in the next one we show that it compresses to the classical Hardy space on the polydisk $\bD^d$, but that it has no non-constant multipliers.

Popa and Vinnikov provided candidates for such a space in \cite{PV18} using asymptotic integral formulae on the Shilov boundary of $\fD_d$. While these spaces have interesting properties of their own, they are not nc RKHS on $\fD_d$ (see \cite[Remark 3.10]{PV18}). We therefore construct a nc RKHS on $\fD_d$ that preserves the symmetries of $\fD_d$ and also compresses onto $H^2(\bD^d)$. 

Let $\bF_d^+$ be the free semi-group on $d$ generators. 
Every $\alpha \in \bF_d^+$ determines a multi-index $\tilde{\alpha} \in \bN_d$ given by $\tilde{\alpha} = (\tilde{\alpha}_1, \ldots, \tilde{\alpha}_d)$, where $\tilde{\alpha}_j$ is equal to the number of times $j$ appears in $\alpha$. 
We shall abuse notation and use the same letter $\alpha$ to denote both a noncommutative word and the multi-index $\tilde{\alpha}$ that corresponds to it.
Define the \emph{nc Hardy space on $\fD_d$} as
\[
H^2(\fD_d) = \left\{ f \sim \sum_{\alpha \in \bF_d^+} c_\alpha Z^\alpha : \sum_{\alpha \in \bF_d^+} \frac{|\alpha|!}{\alpha !} |c_\alpha|^2 < \infty \right\},
\]
where $|\alpha| !$ and $\alpha !$ are the multinomial expressions when $\alpha$ is regarded as an element of $\bN^d$. 
Note that it is not obvious that such a formal nc power-series converges on $\fD_d$. We provide a short proof for the sake of completeness.

\begin{proposition} \label{prop_convergence_of_nc_powerseries_on_D_d}
Let $f \sim \sum_{\alpha} c_\alpha Z^\alpha$ be a formal nc power-series such that
\[
\sum_{\alpha \in \bF_d^+} \frac{|\alpha| !}{\alpha !} |c_\alpha|^2 < \infty.
\]
Then, $f$ converges absolutely on $\fD_d$. 
Furthermore if $X \in \fD_d$ with $ \|X\|_\infty \leq r < 1 $, then 
\[
\|f(X)\| \leq \left( \frac{1}{1-r} \right)^d \left( \sum_{\alpha \in \bF_d^+} \frac{|\alpha!|}{\alpha!} |c_\alpha|^2 \right)^{1/2}.
\]
\end{proposition}

\begin{proof}
Let $f$ be as in the hypothesis, and fix $X \in \fD_d$ with $ \|X\|_\infty \leq r < 1 $ for some $r > 0$. 
Note that
\begin{equation} \label{equation_hardy_space_pseries_convergence_1}
 \sum_{\alpha \in \bF_d^+} \| c_\alpha X^\alpha \| = \sum_{\alpha \in \bF_d^+} |c_\alpha| \, \| X^\alpha \| \leq \sum_{k = 0}^{\infty} \left( \sum_{|\alpha| = k} |c_\alpha| \right) r^k.
\end{equation}

Using the Cauchy-Schwarz inequality, we obtain
\begin{align} \label{equation_hardy_space_pseries_convergence_2}
\sum_{|\alpha| = k} |c_\alpha| &\leq \left( \sum_{|\alpha| = k} \frac{\alpha !}{|\alpha|!} \right)^{1/2} \left( \sum_{|\alpha| = k} \frac{|\alpha| !}{\alpha !} |c_\alpha|^2 \right)^{1/2} \\
\label{equation_hardy_space_pseries_convergence_3}&\leq {k + d - 1 \choose d - 1} \left( \sum_{|\alpha| = k} \frac{|\alpha| !}{\alpha !} |c_\alpha|^2 \right)^{1/2}.
\end{align}
(\ref{equation_hardy_space_pseries_convergence_3}) holds because $|\alpha|!/\alpha!$ is precisely the number of nc monomials that reduce to the same commutative monomial $x^\alpha$ (or, if you prefer, the number of free words that give rise to the same multi-index), and there are ${k + d - 1 \choose d - 1}$ many distinct commutative monomials of length $k$. Then, we get (\ref{equation_hardy_space_pseries_convergence_3}) using the trivial inequality
\[
{k + d - 1 \choose d - 1}^{1/2} \leq {k + d - 1 \choose d - 1}.
\]
Combining (\ref{equation_hardy_space_pseries_convergence_1}) and (\ref{equation_hardy_space_pseries_convergence_3}) gives us 
\[
\sum_{\alpha \in \bF_d^+} \| c_\alpha X^\alpha\| \leq \left( \frac{1}{1-r} \right)^d \left( \sum_{\alpha \in \bF_d^+} \frac{|\alpha| !}{\alpha !} |c_\alpha|^2 \right)^{1/2}.
\] 
Therefore, $f$ converges absolutely on $\fD_d$. 
The second part follows easily from the above bound and the fact that $\|f(X)\| \leq \sum_{\alpha} \| c_\alpha X^\alpha \|$.
\end{proof}

The first part of Proposition \ref{prop_convergence_of_nc_powerseries_on_D_d} shows that $H^2(\fD_d)$ is indeed a Hilbert space of nc functions defined on $\fD_d$ with the inner product
\[
\left\langle\sum_{\alpha \in \bF_d^+} c_\alpha Z^\alpha , \sum_{\alpha \in \bF_d^+} d_\alpha Z^\alpha \right\rangle = \sum_{\alpha \in \bF_d^+} \frac{|\alpha| !}{\alpha !} \, \overline{d_\alpha} c_\alpha.
\]
The second part of Proposition \ref{prop_convergence_of_nc_powerseries_on_D_d} then shows that the point evaluations \[
f \mapsto \langle f(X)v, u \rangle, \, \forall f \in H^2(\fD_d)
\] are bounded for all $X \in \fD_d(n)$, and $v, u \in \bC^n$, $n \in \bN$. Therefore, $H^2(\fD_d)$ is a nc RKHS.
Since $\{ \left( \alpha!/|\alpha|! \right)^{1/2} Z^\alpha \}$ is an orthonormal basis of $H^2(\fD_d)$, we can use \cite[Theorem 3.5]{BMV16} to get that the nc reproducing kernel $K$ of $H^2(\fD_d)$ is given by
\begin{equation} \label{equation_kernel_of_H2(D_d)}
K(X,W)[T] = \sum_{\alpha \in \bF_d^+} \frac{\alpha!}{|\alpha|!} X^\alpha T W^{*\alpha}
\end{equation}
for every $X \in \fD_d(n), W \in \fD_d(m)$ and $T \in M_{n \times m}(\bC)$.

\subsection{Lack of multipliers in $H^2(\fD_d)$} \label{subsection_lack_of_multipliers_on_H2(D_d)}

So far we have used the combinatorial factor $|\alpha| !/\alpha !$ in the definition of the norm of $H^2(\fD_d)$ as a means to obtain convergence on $\fD_d$. We shall now show that $H^2(\fD_d)$ compresses to $H^2(\bD^d)$ in a natural way as a RKHS, however we do not recover the multipliers of $H^2(\bD^d)$.

First, we show that the combinatorial factor is invariant under permutations. 
For a given free word $\alpha = \alpha_1 \alpha_2 \dots \alpha_k \in \bF_d^+$ and some permutation $\sigma \in S_k$, we let $\sigma(\alpha) \in \bF_d^+$ denote 
\[
\sigma(\alpha) := \alpha_{\sigma(1)} \alpha_{\sigma(2)} \dots \alpha_{\sigma(k)}.
\] 
For each $\alpha \in \bF_d^+$, it is then clear that \[
\frac{|\alpha|!}{\alpha !} = \frac{|\sigma(\alpha)| !}{\sigma(\alpha)!}, \, \forall \sigma \in S_{|\alpha|}.
\] Secondly, exchange of any two variables is clearly a unitary transformation on $H^2(\fD_d)$, and $H^2(\fD_d)$ is also clearly invariant under rotations of each variable. Therefore, $H^2(\fD_d)$ preserves the natural symmetries of $\fD_d$.

The \emph{symmetrized nc Hardy space} $H^2_+(\fD_d)$ is defined to be the span of \[
\xi_{\alpha} := \frac{1}{|\alpha|!} \sum_{\sigma \in S_{|\alpha|}} Z^{\sigma(\alpha)}, \, \forall \alpha \in \bF_d^+.
\] Note that $\xi_{\alpha} = \xi_{\sigma(\alpha)}$ for each $\alpha \in \bF_d^+$ and $\sigma \in S_{|\alpha|}$. Since we can always find $\sigma \in S_{|\alpha|}$ that turns $\alpha$ into its multi-index form in $\bN^d$, it follows that \[
H^2_+(\fD_d) = \text{span}\{\xi_\alpha : \alpha \in \bN^d\}.
\] 
A simple combinatorial calculation then shows that \[
\langle \xi_\alpha , \xi_\beta \rangle = \delta_{\alpha,\beta}, \, \forall \alpha, \beta \in \bN^d.
\] The following is therefore a unitary map between $H^2(\bD^d)$ and $H^2_+(\fD_d)$: \[
H^2(\bD^d) \ni \sum_{\alpha \in \bN^d} c_\alpha z^\alpha \longleftrightarrow \sum_{\alpha \in \bN^d} c_\alpha \xi_\alpha \in H^2_+(\fD_d).
\]

Note that all separable Hilbert spaces are isomorphic, but the interesting feature of these symmetrized nc function spaces is the compression of the \emph{nc (left) shift operators} to the commutative shift operators with the help of the above unitary. See \cite[Section (41)]{Sha15} for an exploration of this idea in the case of the Drury-Arveson space $H^2_d$ on the unit row ball $\bB_d$. We now show that this idea does not work for $H^2(\fD_d)$ and $H^2(\bD^d)$.

Let $P_+ : H^2(\fD_d) \rightarrow H^2_+(\fD_d)$ be the projection map, and note that \[
P_+(Z^\alpha) = \xi_\alpha, \, \forall \alpha \in \bF_d^+.
\] It then follows from (\ref{equation_kernel_of_H2(D_d)}) that for each $w \in \fD_d(1) = \bD^d$, the kernel functions $K_{w,v,y}$ of $H^2(\fD_d)$ get compressed to the kernel functions $k_w$ of $H^2(\bD^d)$ given by \[
k_w(z) = \sum_{\alpha \in \bN^d} z^\alpha \overline{w}^\alpha, \, \forall z \in \bD^d.
\]

One might wish to compress the nc (left) shift operators $S_j$, defined formally by \[
S_jf(Z) := Z_j f(Z), \, \forall f \in H^2(\fD_d),
\] to the commutative shift operators on $ H^2(\bD^d) = H^2_+(\fD_d)$ under $P_+$ for each $1 \leq j \leq d$. However, as the following example shows, these nc shifts are not even well-defined on $H^2(\fD_d)$.

\begin{example}
Consider \[
f(Z) = \sum_{k = 0}^\infty \frac{1}{k+1} Z_1^k.
\] Clearly, $f \in H^2(\fD_d)$. The formal nc power-series for $Z_2 f(Z)$ is given by \[
Z_2 f(Z) = \sum_{k = 0}^\infty \frac{1}{k+1} Z_2 Z_1^k.
\] Note that $Z_2 f(Z) \not\in H^2(\fD_d)$ since \[
\|Z_2 f(Z)\|^2 = \sum_{k = 0}^\infty \frac{(k+1)!}{1! \, k!} \left\rvert \frac{1}{k+1} \right\rvert^2 = \sum_{k = 0}^\infty \frac{1}{k+1} = \infty.
\]
\end{example}

In contrast, $\mlt(\cH_d^2) = H^{\infty}(\fB_d)$ can be identified with the \emph{nc analytic Toeplitz algebra} $\fL_d$ generated by the nc shifts on $\fB_d$.

We now show that $\mlt(H^2(\fD_d))$ has no non-trivial elements using a similar argument.

\begin{theorem} \label{thm_multiplier_algebra_nc_hardy_space}
For $d \geq 2$, $\mlt(H^2(\fD_d)) = \{ c1 : c \in \bC \}$ where $1$ is the constant function.
\end{theorem}

\begin{proof}
It is clear that $\{c1 : c \in \bC \} \subseteq \mlt(H^2(\fD_d))$. For the reverse inclusion, suppose $\phi \in \mlt(H^2(\fD_d))$ is non-constant. As $1 \in H^2(\fD_d)$ we get that $\phi = M_{\phi} 1 \in H^2(\fD_d)$. Thus, $\phi(Z) = \sum_\alpha c_\alpha Z^\alpha$ for some constants $c_{\alpha}$. As $\phi$ is non-constant, there exists $0 \neq \alpha_0 \in \bF_d^+$ such that $c_{\alpha_0} \neq 0$. Suppose $\alpha_0 = (a_1, a_2, \dots, a_d) \in \bN^d$ in its multinomial form. Since $d \geq 2$, we can always pick $1 \leq j_0 \leq d$ such that $a_{j_0} < |\alpha_0|$. Note that for all $k \in \bN$, we then have
\[
\frac{\left\|M_{\phi}Z_{j_0}^k\right\|^2}{\left\|Z_{j_0}^k\right\|^2} = \left\|M_{\phi}Z_{j_0}^k\right\|^2 \geq \left \lvert c_{\alpha_0}\right \rvert^2 \left \| Z^{\alpha_0} Z_{j_0}^k \right \|^2 = \left \lvert c_{\alpha_0}\right\rvert^2 \frac{(k+|\alpha_0|)!}{a_1! \dots (k + a_{j_0})! \dots a_d!}.
\]
By the choice of $j_0$ we get that
\[
\left \|M_\phi\right \|^2 \geq \sup_{k \in \bN} \frac{\left \| M_\phi Z_{j_0}^k \right \|^2}{\left \| Z_{j_0}^k \right \|^2} \geq \frac{\left\lvert c_{\alpha_0} \right\rvert^2 a_{j_0}!}{a_1 ! \dots a_d !} \, \sup_{k \in \bN}\frac{\left (k + |\alpha_0|\right )!}{\left (k + a_{j_0}\right )!} = \infty,
\]
which is absurd. Therefore, $\mlt(H^2(\fD_d)) \subseteq \{c1 : c \in \bC\}$ and the result holds.
\end{proof}

This is interesting because, as is well-known, $\mlt(H^2(\bD^d)) = H^{\infty}(\bD^d)$ which is as large a (commutative) multiplier algebra as one can be, but $\mlt(H^2(\fD_d))$ is as small as it can be.
Lastly, we show that even $A(\fD_d)$ has a complicated operator algebraic structure.

\subsection{The C*-envelope of $A(\fD_d)$} \label{subsection_c*_envelope_of_nc_polydisk_algebra}

In this section we shall show that the C*-envelope of the operator algebra $A(\fD_d)$ is equal to the full C*-algebra of the free group. We start by giving a basic introduction to C*-envelopes. The reader is advised to refer to \cite{PaulsenBook} for more details on the following discussion.

Let $\cA$ be a unital operator algebra, and let $\cB = C^*(\cA)$ be a C*-algebra generated by $\cA$. An ideal $\cI \triangleleft \cB$ is said to be a \emph{boundary ideal} for $\cA$ in $\cB$ if the restriction of the quotient map $\pi : \cB \to \cB / \cI$ to $\cA$ is completely isometric. The \emph{Shilov ideal}, denote by $\cJ \triangleleft \cB$, is the unique largest boundary ideal for $\cA$ in $\cB$. The \emph{C*-envelope} of $\cA$ is defined to be the quotient \[
C^*_e(\cA) := \cB / \cJ.
\]
Even though $C^*_e(\cA)$ is defined using a C*-algebra generated by $\cA$, it is independent of the realization $\cA \subseteq \cB$. This follows from the following universal property of the C*-envelope: if $i : \cA \to \cB'$ is a completely isometric homomorphism such that $\cB' = C^*(i(\cA))$, then there exists a unique surjective $*$-homomorphism $\rho : \cB' \to C^*_e(\cA)$ satisfying \[
\pi(a) = \rho(i(a)), \, \forall a \in \cA.
\] In essence, $C^*_e(\cA)$ is the `smallest' C*-algebra that is generated by a completely isometrically isomorphic copy of $\cA$. Let us now turn to identifying $C^*_e(A(\fD_d))$. We set \[
\ol{\fD}_d^\infty := \{T \in \cL(\cH)^d : \|T_j\|\leq 1, \, \forall \, 1 \leq j \leq d \textrm{ and } \cH \textrm{ -- Hilbert space}\},
\] and
\[
\cU_d := \{ U \in \cL(\cH)^d : U_j \text{ is a unitary, } \forall \, 1 \leq j \leq d \text{ and } \cH \text{ -- Hilbert space} \}.
\]

\begin{lemma}\label{lem:norms}
For every matrix valued free nc polynomial $p \in \bC\langle Z \rangle \otimes M_n$, \[
\sup_{T \in \fD_d} \|p(T)\| = \sup_{T \in \ol{\fD}_d} \|p(T)\| = \sup_{T \in \ol{\fD}_d^\infty} \|p(T)\| = \sup_{U \in \cU_d} \|p(U)\| .
\]
\end{lemma}
\begin{proof}
The first equality is clear. 
In the second equality, an inequality `$\leq$' evidently holds. 
To see that it is an equality, fix $p \in \bC\langle Z \rangle \otimes M_n$ and $T \in \cL(\cH)^d$ with $\|T_j\| \leq 1$ for all $1 \leq j \leq d$. 
Then for every finite dimensional $\cG \subseteq \cH$, we can identify the compressed tuple $T_\cG:=P_\cG T \vert_\cG$ as an element in $\ol{\fD}_d$. 
The family $\{T_\cG\}$ forms a bounded net that converges strongly to $T$, and it follows that for every $h \in \cH \otimes \bC^n$, \[
\left \|p(T)h \right \| \leq \sup_{\cG} \left \|p(T_\cG) P_\cG h \right \| \leq \sup_{\cG} \left  \|p(T_\cG)\right \| \left \|h\right \| 
\] as required. As for the third and last equality, certainly there is an inequality `$\geq$' since unitaries are in particular contractions. But for every $T \in \ol{\fD}_d^\infty$ there is a unitary dilation i.e., a tuple of unitaries $U \in \cL(\cK)^d$ where $\cK \supseteq \cH$ such that for every $p\in \bC\langle Z \rangle$, we have
\[
p(T) = p_\cH p(U)\vert_\cH. \] Indeed we simply construct $U_j$ as a unitary dilation of $T_j$ on $\cK$ independently for each $j$. 
Then $\|p(T)\| \leq \|p(U)\|$, and this clearly persists for matrix valued polynomials as well. 
\end{proof}

Let $C^*(\bF_d)$ be the full group C*-algebra of $\bF_d$. 
Let us write $u = (u_1, \ldots, u_d)$ for the tuple of canonical generators of $C^*(\bF_d)$. 
By the universal property of $C^*(\bF_d)$, every unitary tuple $U = (U_1, \ldots, U_d)$ is the image of $u$ under a $*$-representation, and thus \[
\|p(u)\| = \sup_{U \in \cU_d} \|p(U)\|, \, \forall p \in \bC \langle Z \rangle \otimes M_n.
\] Therefore, we have the following consequence. 

\begin{theorem} \label{thm:C*-envelope_A(D_d)}
There is a completely isometric isomorphism $\varphi : A(\fD_d)  \to \ol{\alg}(u_1, \ldots, u_d)$ determined by $\varphi(Z_j) = u_j$, for all $j=1, \ldots, d$. 
Consequently, $C^*_e(A(\fD_d)) = C^*(\bF_d)$. 
\end{theorem}
\begin{proof}
By Lemma \ref{lem:norms} and the above observation, $\varphi$ is a completely isometric isomorphism. 
Therefore, $C^*_{e}(A(\fD_d)) = C^*_{e}(\ol{\alg}(u_1, \ldots, u_d))$, since the C*-envelope is an invariant of the operator algebra structure.
We know that if an operator algebra $\cA \subseteq \cB = C^*(\cA)$ is generated by unitaries, then the Shilov boundary of $\cA$ in $\cB$ is trivial (see \cite[Corollary 2.2.8]{Arv69}). Thus, 
\[
C^*_{e}(A(\fD_d)) = C^*_{e}(\ol{\alg}(u_1, \ldots, u_d)) = C^*(\bF_d). 
\] 
Alternatively, one can use the fact that unitaries are hyperrigid, meaning that a unital complete isometry between $\operatorname{span}\{1, u_1, \ldots, u_d\} \subset C^*(\bF_d)$ and its image in $ C^*_{e}(\ol{\alg}(u_1, \ldots, u_d))$ must extend to a $*$-isomorphism (see \cite[Theorem 3.3]{Arv11} or \cite[Corollary 4.10]{Kiri}). 
\end{proof}

Theorem \ref{thm:C*-envelope_A(D_d)} suggests that searching for a natural and tractable representation of $A(\fD_d)$ on a Hilbert space of nc functions might also be a challenging task, since in any such representation the C*-algebra generated by $A(\fD_d)$ will have $C^*(\bF_d)$ as a quotient. 
However, $C^*(\bF_d)$ is well known to be a complicated universal object. 
This should be compared with the C*-envelope of $A(\fB_d)$, which is the Cuntz algebra $\cO_d$ --- a simple and nuclear C*-algebra!


\section{\bf Bounded NC Functions on NC Operator Balls} \label{section_nc_operator_balls}

The observations in Section \ref{section_nc_function_theory_in_polydisk} motivate us to employ more general nc function theory techniques that will work for the nc unit polydisk and more general nc domains, namely \emph{nc operator balls}. 
We showcase the important properties of these domains in this section. 
We shall follow closely the notations in \cite{BMV18}.

\subsection{Full NC sets and interpolation} \label{subsection_full_nc_sets}

A nc subset $\Omega \subset \bM^d$ is said to be \emph{full} if it is closed under left injective intertwiners i.e., 
\[
X \in \Omega, \, \widetilde{X} \in \bM^d \text{ s.t. } I\widetilde{X} = X I \text{ for some } I \text{ -- injective} \Rightarrow \widetilde{X} \in \Omega.
\] 
In particular, $\Omega = \bM^d$ is a full nc set. 
Let $\cL(\cR,\cS)_{nc}$ be as in (\ref{equation_ambient_operator_space}) for some Hilbert spaces $\cR$, $\cS$. 
A nc operator ball $\bD_Q$, as defined in \eqref{eq:operator_ball}, is a natural domain for interpolation on nc subsets $\fX$ that are \emph{relatively full} in $\bD_Q$ i.e., \[
X \in \fX, \, \widetilde{X} \in \bD_Q \text{ s.t. } I\widetilde{X} = XI \text{ for some } I \text{ -- injective} \Rightarrow \widetilde{X} \in \fX.
\]

The most important tool available in this regard is \cite[Corollary 3.4]{BMV18}, which we state in terms of our notation for the reader's convenience.

\begin{theorem} \label{theorem_BMV_interpolation_full_sets}
	
Let $Q$ and $\bD_Q$ be as in (\ref{eq:operator_ball}). Suppose $\cU, \cV$ are some other coefficient Hilbert spaces, and $\mathfrak{X} \subset \bD_Q$ is a relatively full nc subset of $\bD_Q$. Then, a nc map $S_0 : \mathfrak{X} \rightarrow \cL(\cU,\cV)_\text{nc}$ can be extended to a nc map $S : \bD_Q \rightarrow \cL(\cU,\cV)_{\text{nc}}$ without changing the sup-norm i.e., \[
S\rvert_{\fX} = S_0, \text{ and } \sup_{X \in \mathfrak{X}} \| S_0(X) \| = \sup_{X \in \bD_Q} \| S(X) \|.
\]
	
\end{theorem}

We also have a version of the maximum modulus principle for nc operator balls, which we shall refer to as the \emph{boundary value principle} for $\bD_Q$ henceforth. Let $\overline{\bD_Q} = \sqcup_{n = 1}^\infty \overline{\bD_Q(n)}$ be the level-wise closure of $\bD_Q$, and note that \[
\overline{\bD_Q} = \{ X \in \bM^d \, : \, \| Q(X) \| \leq 1 \}.
\] We denote the topological boundary of $\bD_Q$ by $\partial \bD_Q := \overline{\bD_Q} \setminus \bD_Q$. Clearly, \[
\partial \bD_Q = \{ X \in \bM^d \, : \, \| Q(X) \| = 1 \}.
\]

\begin{theorem}[Boundary value principle] \label{theorem_bdy_value_principle}

Let $\Omega \subseteq \bM^l, \, l \in \bN$ be a nc domain, and let $Q$ and $\bD_Q$ be as in (\ref{eq:operator_ball}). 
If $F : \Omega \rightarrow \overline{\bD_Q}$ is a nc map such that $F(X_0) \in \partial \bD_Q$ for some $X_0 \in \Omega$, then $F(\Omega) \subseteq \partial \bD_Q$.

In particular, if $F : \Omega \rightarrow \overline{\bD_Q}$ is a nc map such that $F(X) \in \bD_Q$ for some $X \in \Omega$, then $F(\Omega) \subseteq \bD_Q$.

\end{theorem}

\begin{proof}
Let $F : \Omega \rightarrow \overline{\bD_Q}$ be a nc map such that $F(X_0) \in \partial \bD_Q$ for some $X_0 \in \Omega(n)$, $n \in \bN$. In other words, $\| Q(F(X_0)) \| = 1$. We can therefore find a bounded linear functional $\Lambda \in (\cL(\cR,\cS))^*$ such that $\|\Lambda\| = 1 \text{ and } \| \Lambda(Q(F(X_0))) \| = 1$. We then define a holomorphic map $G : \Omega (n) \rightarrow \overline{\bD}$ by
\[
G(X) = \Lambda(Q(F(X))), \, \forall \, X \in \Omega(n).
\] 
It follows from the maximum modulus principle that $|G(X)| = 1$ for all $X \in \Omega(n)$. Thus,
\[
1 \geq \| Q(F(X)) \| \geq \| \Lambda(Q(F(X))) \| = |G(X)| = 1, \, \forall X \in \Omega(n),
\]
and we get $F(\Omega(n)) \subseteq \partial \bD_Q$.

We now move to the other levels by using the following trick. 
For any $X \in \Omega(n)$, define
\[
X^{(m)} := \oplus_{k=1}^m X  \in \Omega(mn), \, \forall m \in \bN.
\] 
As $Q\circ F$ is a nc map, we get \[
\|Q(F(X^{(m)}))\| = \|Q(F(X))\|= 1.
\] 
Repeating the previous argument shows that $\|Q(F(Y))\| = 1$ for any $Y \in \Omega(mn)$. 
Let $W \in \Omega(m)$, and let $W^{(n)}$ be the direct sum of $W$ with itself $n$ times. 
Then
\[
\|Q(F(W))\| = \|Q(F(W^{(n)}))\| = 1
\]
and so, $F(W) \in \partial \bD_Q$ as required. The second part of the theorem then follows easily.
\end{proof}

\begin{remark}
It is worth noting that the {\em levelwise} boundary value principle follows from the fact that each level is convex and hence, in particular, it is what is known in the several complex variables literature as a {\em taut} domain. 
Thus, we could have referred to \cite[Proposition 2.1.6]{Abate} but we chose to keep the discussion self-contained.
\end{remark}

\begin{example} \label{example_image_could_lie_in_boundary}
In the case $\bD_Q = \fD_d$, there are non-constant nc maps that map into $\partial \fD_d$. Take, for instance, $F : \fD_2 \rightarrow \partial \fD_2$ defined as 
\[
F(X_1,X_2) = (X_1, I_n), \, \forall (X_1,X_2) \in \fD_2 (n), \, n \in \bN,
\] 
where $I_n$ is the $n \times n$ identity matrix in $M_n$. Hence, we cannot conclude that the map $F$ as in Theorem \ref{theorem_bdy_value_principle} is necessarily constant for a general $\bD_Q$, and thus we may not have a proper maximum modulus principle for such a $\bD_Q$. For $\bD_Q = \fB_d$ however, \cite[Lemma 6.11]{SSS18} shows that such an $F$ as in Theorem \ref{theorem_bdy_value_principle} is always constant.
\end{example}

\subsection{Bounded NC maps on $\bD_Q$} \label{subsection_bounded_nc_maps_on_D_Q}

As in Section \ref{subsection_definitions_notation}, for any nc subset $\mathfrak{X} \subseteq \bD_Q$ we define
\[
H^\infty(\mathfrak{X}) = \left\{f : \fX \rightarrow \bM^1 : f \text{ is nc and } \sup_{X \in \mathfrak{X}} \| f(X) \| < \infty \right \}.
\]
In particular, $H^\infty(\bD_Q)$ is the algebra of all bounded nc functions on $\bD_Q$. Similarly, define
\[
A(\mathfrak{X}) = \overline{\left\{p \rvert_{\mathfrak{X}} : p \in \bC \langle Z \rangle \right\}}^{\|\cdot\|_{\infty}\rvert_{\mathfrak{X}}}.
\]
In particular, $A(\bD_Q)$ is the norm-closed subalgebra of $H^{\infty}(\bD_Q)$ generated by free nc polynomials. We define the \emph{bounded-pointwise topology} on $H^{\infty}(\mathfrak{X})$ to be the weakest topology in which a bounded net $\{F^{(\kappa)}\}_{\kappa \in I} \subset H^{\infty}(\mathfrak{X})$ converges to $F \in H^{\infty}(\mathfrak{X})$ if and only if 
\[
F^{(\kappa)}(X) \xrightarrow{\|\cdot\|_\infty} F(X), \, \forall X \in \mathfrak{X}.
\]

We now show that when $\fX$ is relatively full in $\bD_Q$, functions in $H^{\infty}(\fX)$ have a global power-series representation. Since $Q$ is a linear operator-valued nc map, it follows that $\bD_Q$ is \emph{homogeneous} i.e., $\lambda \bD_Q \subseteq \bD_Q \foral \lambda \in \overline{\bD}$:
\[
X \in \bD_Q, \, \lambda \in \overline{\bD} \implies \|Q(\lambda X)\| = \|\lambda Q(X)\| = |\lambda| \left \|Q(X)\right \| < 1.
\]
Secondly, $\bD_Q$ is open in the \emph{uniformly-open nc topology} (as defined in \cite[Section 7.2]{KVV14}): 

Let $m, n \in \bN$, $X \in \bD_Q(n)$ be arbitrary. If $Y \in \Omega(mn)$ is such that
\[
\left \| Y - \oplus_{k = 1}^m X \right \|_\infty < r,
\]
then since $Q$ is a linear nc map we get \begin{equation} \label{equation_linearity_of_Q}
\|Q(Y)\| \leq \|Q(X)\| + \left \| Q\left( Y - \oplus_{k = 1}^m X \right) \right \|.
\end{equation} Since $\| Q(X) \| < 1 - \delta$ for some $\delta > 0$, we choose $r = \delta/(d \cdot \max_{1 \leq j \leq d} \|Q_j\|)$ so that
\begin{equation} \label{eqn:uniform_openness_of_D_Q}
\| Q(W) \| \leq \left(d \cdot \max_{1 \leq j \leq d}\left \|Q_j \right \| \right) \left \| W \right \|_\infty < \delta, \, \forall \, W \in \bM^d, \, \| W \|_\infty < r.
\end{equation}
As the choice of $r > 0$ is independent of $m \in \bN$, we can plug in $W = Y - \oplus_{k = 1}^m X$ in (\ref{eqn:uniform_openness_of_D_Q}) and use (\ref{equation_linearity_of_Q}) to get that \[
B_{nc}(X,r) := \sqcup_{m = 1}^\infty B\left( \oplus_{k = 1}^m X , r \right) \subset \bD_Q.
\] Since $X \in \bD_Q$ was arbitrarily chosen, this shows that $\bD_Q$ is uniformly-open. \cite[Theorem 7.21]{KVV14} then shows that every $F \in H^{\infty}(\bD_Q)$ has a global power-series representation about $0$:
\[
F(X) = \sum_{\alpha \in \bF_d^+} c_{\alpha} X^{\alpha}, \, \forall X \in \bD_Q.
\]
Furthermore, the power-series converges absolutely and uniformly on $r \bD_Q $ for all $ r < 1$. When $\fX$ is relatively full in $\bD_Q$, by Theorem \ref{theorem_BMV_interpolation_full_sets}, every $f \in H^{\infty}(\fX)$ extends to some $F \in H^{\infty}(\bD_Q)$. Therefore, $f$ inherits a global power-series representation via $F$. Note that such a power-series representation is not necessarily unique since the extension $F$ may not be unique.

The power-series for $F \in H^\infty(\bD_Q)$ provides us with a (formal) \emph{homogeneous expansion} $F = \sum_{k = 0}^\infty F_k$ for each $F \in H^{\infty}(\bD_Q)$, where \[
F_k(X) := \sum_{|\alpha| = k} c_{\alpha} X^{\alpha}, \, \forall X \in \bD_Q, \, k \in \bN \cup \{0\}.
\] The homogeneous expansion will be crucial so we highlight some of its properties below.

\begin{proposition} \label{prop_hom_expansion_cesaro_convergent} For each $k \in \bN \cup \{0\}$ and $F \in H^{\infty}(\bD_Q)$, let $F_k$ denote the $k^{\text{th}}$ term in the homogeneous expansion of $F$.
\begin{enumerate}
\item[$(1)$] The map $P_k : H^\infty(\bD_Q) \rightarrow \bC \langle Z \rangle$ given by $P_k(F) = F_k$ is completely contractive and continuous in the bounded-pointwise topology for each $k \in \bN \cup \{0\}$. 
\item[$(2)$] The Ces\`aro sums $\Sigma_k(F)$ of the homogeneous expansion of every $F \in H^\infty(\bD_Q)$ converge to $F$ in the bounded-pointwise topology.
\item[$(3)$] For every $F \in A(\bD_Q)$, the Ces\`aro sums $\Sigma_k(F)$ of the homogeneous expansion of every $F \in A(\bD_Q)$ converge to $F$ in the norm.
\end{enumerate}
\end{proposition}

\begin{proof}

This proof is similar to the one given in \cite[Lemma 6.1]{DRS11} for the commutative unit-ball $\bB_d$. We provide details below since we consider a much more general setup.

Fix $t \in [0,2\pi]$ and let $\gamma_t : H^\infty(\bD_Q) \rightarrow H^\infty(\bD_Q)$ be the map \[
[\gamma_t(F)](X) = F(e^{it} X), \, \forall X \in \bD_Q.
\] For all $t \in [0, 2 \pi)$, the map $\gamma_t$ is a bounded-pointwise continuous, completely isometric automorphism of $H^\infty(\bD_Q)$ that fixes $A(\bD_Q)$. 
Now, fix $F \in H^\infty(\bD_Q)$ and define a closed curve $\Gamma_F : [0,2\pi) \rightarrow H^\infty(\bD_Q)$ by \[
\Gamma_F(t) = \gamma_t(F), \, \forall t \in [0,2\pi].
\] This map is continuous in the bounded-pointwise topology. 
If $F \in A(\bD_Q)$, then $\Gamma_F$ takes values in $A(\bD_Q)$, and $\Gamma_F$ is norm continuous. 
This allows us to define for each $k \in \bN \cup \{0\}$ the map $\widetilde{P_k} : H^\infty(\bD_Q) \rightarrow H^\infty(\bD_Q)$ given by
\begin{align*}
\widetilde{P_k}(F) &= \frac{1}{2 \pi} \int_0^{2 \pi} \Gamma_F(t) e^{-ikt} dt \\
                   &= \frac{1}{2 \pi} \int_0^{2 \pi} \gamma_t(F) e^{-ikt} dt. 
\end{align*}
The integral is defined as a pointwise integral for $F \in H^\infty(\bD_Q)$ and as a norm convergent integral for $F \in A(\bD_Q)$. 
Note that for all $p \in \bC \langle Z \rangle$ and all $k \in \bN \cup \{0\}$, we have
\[
\widetilde{P_k}(p) = p_k = P_k(p).
\]
$\widetilde{P_k}$ is a complete contraction since it is an average of completely isometric automorphisms. 
By the dominated convergence theorem, it is also bounded-pointwise continuous. 
Therefore $\widetilde{P_k}$ maps into the finite-dimensional (hence closed) subspace $H^{(k)}$ of $\bC \langle Z \rangle$, consisting of homogeneous free nc polynomials of degree $k$. 

Lastly, for any $F \in H^\infty(\bD_Q)$ and any $X \in \bD_Q(n), \, n \in \bN$, the power-series of $F$ converges absolutely and uniformly in some neighborhood $r \bD_Q(n)$ ($r<1)$ containing $X$, which shows \[
\widetilde{P_k}(F)(X) = F_k(X) = P_k(F)(X), \,\forall X \in \bD_Q.
\] Using a standard argument involving the Fej\'er kernel, it then follows that the Ces\`aro sums of the homogeneous expansion of $F \in H^\infty(\bD_Q)$ converge in the bounded-pointwise topology to $F$, and if $F \in A(\bD_Q)$ then they converge in the norm.
\end{proof}

\begin{remark} \label{remark_h_infinity_weak_*_topology}
It is known that $H^\infty(\fB_d)$ is a multiplier algebra $H^{\infty}(\fB_d) = \Mult(\cH^2_d)$ (see \cite[Proposition 3.5]{SSS18}) and so it is endowed with a weak-operator topology. 
For bounded nets, convergence in the weak-operator topology is equivalent to pointwise convergence \cite[Lemma 2.5]{SSS18}. 
Being a multiplier algebra, $H^\infty(\fB_d)$ is also a dual algebra \cite[Corollary 2.6]{SSS18} with a weak-$*$ topology, and it is known that the Ces\`aro sums of the homogeneous expansion of every $F \in H^{\infty}(\fB_d)$ converge to $F$ in the weak-* topology \cite[Theorem 3.11]{SSS18}.
The weak-$*$ and weak-operator topology have been shown to coincide in the case of $H^\infty(\fB_d)$ by Davidson and Pitts \cite{DavPitts99}. 
As we saw in Theorem \ref{thm:mult_poly}, the bounded nc functions on a general nc operator ball $\bD_Q$ might not be naturally represented as a multiplier algebra, and so for our purposes here we concentrate on the bounded-pointwise topology on $H^{\infty}(\bD_Q)$.
One can describe it in terms of functionals but we do not pursue this here. 
\end{remark}

\subsection{Varieties and ideals} \label{subsection_varieties_and_homogeneous_ideals}

We shall now establish certain properties of homogeneous subvarieties in $\bD_Q$, and homogeneous ideals in $A(\bD_Q)$ and $H^{\infty}(\bD_Q)$. These properties will be essential when we obtain a homogeneous Nullstellensatz for these algebras. Our goal is to find a correspondence between homogeneous ideals and homogeneous varieties. For general nonhomogeneous varieties, such a correspondence is not possible, even for $\bD_Q = \fB_d$. We refer the reader to the discussion following \cite[Theorem 7.3]{SSS18} for more details.

Let $(\cA, \tau)$ denote either $A(\bD_Q)$ equipped with the norm topology, or $H^{\infty}(\bD_Q)$ equipped with the bounded-pointwise topology. We say that $\fV \subset \bD_Q$ is a {\em (holomorphic) subvariety} in $\bD_Q$ if it is the joint zero-set of some collection of functions in $\cA$. If one considers the joint zero-set of a collection of free nc polynomials then $\fV$ is said to be an {\em algebraic subvariety}. A subvariety $\fV$ (holomorphic or algebraic) is said to be {\em homogeneous} if $\lambda \fV \subseteq \fV$ for all $\lambda \in \overline{\bD}$. Given a holomorphic subvariety $\fV \subset \bD_Q$, the corresponding ideal $I(\fV) \triangleleft \cA$ is defined as
\begin{equation} \label{eqn:ideal_of_V_definition}
I(\fV) = \left\{ F \in \cA : F(X) = 0, \foral X \in \fV \right\}.
\end{equation}
One can similarly define a corresponding ideal in $\bC\langle Z \rangle$ by replacing $F \in \cA$ with $p \in \bC\langle Z \rangle$ in (\ref{eqn:ideal_of_V_definition}); note that we use the slightly different notation $I_p(\cdot)$ to avoid confusion:
\[
I_p(\fV) = \left\{ p \in \bC\langle Z \rangle : p(X) = 0, \foral X \in \fV \right\}.
\]
An ideal $\cJ \triangleleft \cA$ or $\cJ \triangleleft \bC \langle Z \rangle$ is said to be {\em homogeneous} if for all $F \in \cJ$, every homogeneous part $F_k$ of $F$ is also in $\cJ$. 
For a given ideal $\cJ \triangleleft \cA$ or $\cJ \triangleleft \bC \langle Z \rangle$, its corresponding variety $V(\cJ)$ is defined as
\[
V(\cJ) = \left\{ X \in \bD_Q : F(X) = 0, \foral F \in \cJ \right\}.
\]

Let $F \in \cA$. For $r < 1$ we let $F_r$ be the nc function $X \mapsto F(rX)$. 
It is clear that $F_r \in \cA$ and that $\|F_r\| \leq \|F\|$. 
By Proposition \ref{prop_hom_expansion_cesaro_convergent}, it follows that the homogeneous expansion of $F_r$ is given by $F_r = \sum_k r^k F_k$, and that this series is norm convergent to $F_r$ in $\cA$. In particular, the series is $\tau$-convergent to $F_r$ in $\cA$. This allows us to obtain the following characterization.

\begin{proposition} \label{prop_homogeneous_ideal_equivalence}

Let $\cJ \triangleleft \cA$ be a $\tau$-closed ideal. Then, $\cJ$ is homogeneous if and only if for every $F \in \cJ$ and $r < 1$, we have $F_r \in \cJ$.

\end{proposition}

\begin{proof}
Let $\cJ$ be a $\tau$-closed homogeneous ideal in $\cA$ and let $F \in \cJ$. As $F_k \in \cJ$ for each $k \in \bN \cup \{0\}$, and $\sum_k r^k F_k$ is $\tau$-convergent to $F_r$ in $\cA$ for all $r < 1$, we get that $F_r \in \cJ$.

Conversely, suppose $\cJ$ is a $\tau$-closed ideal in $\cA$ such that $F_r \in \cJ$ for all $r < 1$ and $F \in \cJ$. Let $F \in \cJ$ and note that since $\cJ$ is a proper ideal, it follows (by letting $r = 0$) that $F_0 = 0 \in \cJ$.
We can then inductively show that for all $l \in \bN \cup \{0\}$, we have
\begin{align*}
F_{l + 1} &= \lim_{r \rightarrow 0} \sum_{k = 0}^{\infty} r^k F_{k + l + 1} \\
&= \lim_{r \rightarrow 0} \frac{F_r - \sum_{k = 0}^l r^k F_k }{r^{l + 1}} \in \cJ.
\end{align*}
This completes the proof.
\end{proof}

Using the definitions and Proposition \ref{prop_homogeneous_ideal_equivalence} it is easy to see that, just as in the algebraic case, there is a natural connection between homogeneous ideals and varieties.

\begin{proposition} \label{prop_homogeneous_ideals_vs_varieties}
	
If $\fV \subset \bD_Q$ is a homogeneous variety then the corresponding ideal $I(\fV) \triangleleft \cA$ is homogeneous as well. Similarly, if $\cJ \triangleleft \cA$ or $\cJ \triangleleft \bC\langle Z \rangle$ is a homogeneous ideal then the corresponding variety $V(\cJ)$ is homogeneous as well.
	
\end{proposition}

Proposition \ref{prop_homogeneous_ideal_equivalence} is also valid for $\bC \lel Z \rir$, the proof of which is known and similar to ours.



\section{\bf The Algebras $\cA(\fV)$ and Quotients of $(\cA,\tau)$} \label{section_cAV_as_quotients}

The purpose of this section is to identify the algebras $A(\fV)$ and $H^\infty(\fV)$ as quotients of the parent algebras $A(\bD_Q)$ and $H^\infty(\bD_Q)$. 
For the algebras of the form $A(\fV)$ we need to require that $\fV$ is a homogeneous subvariety.
We begin with a homogeneous Nullstellensatz that clarifies the nature of the quotient in the homogeneous case.

\subsection{Homogeneous Nullstellensatz for $(\cA, \tau)$} \label{subsection_homogeneous_Nullstellensatz}

In order to show that a homogeneous Nullstellensatz holds for $(\cA, \tau)$ we need one last ingredient.

\begin{proposition} \label{prop_homogeneous_ideal_closure_property}

Given a $\tau$-closed ideal $\cJ \triangleleft \cA$, the following are equivalent:

\begin{enumerate}

\item $\cJ$ is homogeneous.

\item $\cJ = \overline{\cI}^{\tau}$ for some homogeneous $\cI \triangleleft \bC \langle Z \rangle$.

\end{enumerate}

\noindent Moreover, $\cJ \cap \bC \langle Z \rangle$ is the unique ideal in $\bC \langle Z \rangle$ whose closure in $\cA$ is $\cJ$.

\end{proposition}

\begin{proof}

$(1) \Rightarrow (2)$ : Suppose $\cJ \triangleleft \cA$ is $\tau$-closed and homogeneous. 
It is easy to check that $\cI := \cJ \cap \bC \lel Z \rir$ is a homogeneous ideal in $\bC \lel Z \rir$. 
The inclusion $\overline{\cI}^{\tau} \subseteq \cJ$ is clear, since $\cJ$ is $\tau$-closed. 
For the reverse inclusion, let $F = \sum_k F_k \in \cJ$. Note that $F_k \in \cJ$ for each $k \in \bN \cup \{0\}$, and thus $F_k \in \cI$ as well. This implies that the Ces\`aro sums $\Sigma_k(F)$ of the homogeneous expansion of $F$ are in $\cI$ for each $k$, and thus $F \in \overline{\cI}^{\tau}$ by Proposition \ref{prop_hom_expansion_cesaro_convergent}.

$(2) \Rightarrow (1) $ : Let $\cJ = \overline{\cI}^{\tau}$ for some homogeneous $\cI \triangleleft \bC \lel Z \rir$. 
Given $F = \sum_k F_k \in \cJ$, let $\{p^{(\kappa)}\}_{\kappa \in I} \in \cI$ be a net that is $\tau$-convergent to $F$. Using Proposition \ref{prop_hom_expansion_cesaro_convergent}, we get
\[
p^{(\kappa)}_k = P_k(p^{(\kappa)}) \rightarrow P_k(F) = F_k, \,\forall k \in \bN \cup \{0\}.
\]
Since $p^{(\kappa)}_k \in \cI$ for each $\kappa$ and $k$, we get that $F_k \in \overline{\cI}^{\tau} = \cJ$ for each $k$.

It only remains to show that $\cJ \cap \bC \langle Z \rangle$ is the unique such ideal in $\bC \langle Z \rangle$. 
Suppose $\cJ_0 \triangleleft \bC \langle Z \rangle$ is another homogeneous ideal such that $\overline{\cJ_0} = \cJ$, then note that $V(\cJ_0) = V(\cJ)$.
Now, $\cJ \cap \bC \langle Z \rangle$ is also a homogeneous ideal such that
\[
V(\cJ \cap \bC \langle Z \rangle) = V(\cJ).
\]
Choosing $r > 0$ small enough (as in (\ref{eqn:uniform_openness_of_D_Q})) so that $r \fB_d \subset r \fD_d \subset \bD_Q$, we get
\[
V(\cJ_0) \cap r \fB_d = V(\cJ) \cap r \fB_d = V(\cJ \cap \bC \langle Z \rangle) \cap r \fB_d.
\]
Since $\cJ_0$ and $\cJ \cap \bC \lel Z \rir$ are both homogeneous ideals in $\bC \lel Z \rir$, it follows from Proposition \ref{prop_homogeneous_ideals_vs_varieties} that $V(\cJ_0)$ and $V(\cJ \cap \bC \lel Z \rir)$ are homogeneous varieties. Therefore,
\[
V_{\fB_d}(\cJ_0) = V_{\fB_d}(\cJ \cap \bC \lel Z \rir)
\]
where for all $\cI \triangleleft \bC \lel Z \rir$,
\[
V_{\fB_d}(\cI) := \{ X \in \fB_d : p(X) = 0, \, p \in \cI \}.
\]
By the homogeneous Nullstellensatz for $\bC \langle Z \rangle$ in $\fB_d$ (see \cite[Theorem 7.3]{SSS18}), we get that
\[
\cJ_0 = I(V_{\fB_d}(\cJ)) = \cJ \cap \bC \langle Z \rangle.
\]
This completes the proof.
\end{proof}

\begin{remark} \label{remark_algebraic = holomorphic_homogoneous_varieties}

Proposition \ref{prop_homogeneous_ideal_closure_property} shows that homogeneous holomorphic varieties are the same as homogeneous algebraic varieties in $\bD_Q$, since for any $\cJ \triangleleft \cA$ we have
\[
V(\cJ) = V(\cJ \cap \bC \lel Z \rir).
\]

\end{remark}

We are now ready to prove the homogeneous Nullstellensatz for $(\cA, \tau)$.

\begin{theorem}[Homogeneous Nullstellensatz] \label{thm_nullstellensatz_for_homogenous_poyhedra}

$(1)$ If $\cJ \triangleleft \bC \lel Z \rir$ is a homogeneous ideal, then
\[
I_p(V(\cJ)) = \cJ.
\]

$(2)$ If $\cJ \triangleleft \cA$ is a homogeneous ideal, then
\[
I(V(\cJ)) = \overline{\cJ}^{\tau}.
\]

\end{theorem}

\begin{proof}

$(1)$ Proposition \ref{prop_homogeneous_ideals_vs_varieties} shows that $V(\cJ)$ is a homogeneous variety and $I_p(V(\cJ))$ is a homogeneous ideal. Also, by the homogeneous Nullstellensatz for subvarieties of $\fB_d$ \cite[Theorem 7.3]{SSS18}, we know that $I_p(V_{\fB_d}(\cJ)) = \cJ$. Therefore, it suffices to show that $I_p(V(\cJ))$ and $I_p(V_{\fB_d}(\cJ))$ contain the same collection of homogeneous polynomials.

Let $r > 0$ be small enough as in (\ref{eqn:uniform_openness_of_D_Q}) so that $r \fB_d \subset \bD_Q$, and let $p \in I_p\left(V_{\fB_d}\left(\cJ \right)\right)$ be an arbitrary homogeneous polynomial of degree $k \in \bN$. Thus, for every $X \in r \fB_d \cap V(\cJ)$ we get $p(X) = 0$. Given any $X \in V(\cJ)$, let $\lambda \in \overline{\bD} \setminus \{0\}$ be such that $\lambda X \in r \fB_d$. We then get
\[
p(X) = \frac{1}{\lambda^k} p(\lambda X) = 0.
\]
Therefore, $p \in I_p(V(\cJ))$ and we obtain the inclusion
\[
I_p(V_{\fB_d}(\cJ)) \subseteq I_p(V(\cJ)).
\]
For the reverse inclusion, let $p \in I_p(V(\cJ))$ be a homogeneous polynomial of degree $k \in \bN$. Thus, for every $X \in r \fB_d \cap V(\cJ)$ we get $p(X) = 0$. By the homogeneity of $p$, it then follows that $p(X) = 0$ for every $X \in \fB_d$. Thus, $p \in I_p(V_{\fB_d}(\cJ))$ and we get the required inclusion
\[
I_p(V(\cJ)) \subseteq I_p(V_{\fB_d}(\cJ)).
\]

$(2)$ Note that for any $\cJ \triangleleft \cA$, we have $V(\cJ) = V(\overline{\cJ}^{\tau})$ and $I(V(\cJ)) = I\left(V(\overline{\cJ}^{\tau})\right)$.
Therefore, without loss of generality, we can assume $\cJ \triangleleft \cA$ is a $\tau$-closed homogeneous ideal. Let $\cJ_0 := \cJ \cap \bC \lel Z \rir$ and $\cI_0 := I(V(\cJ)) \cap \bC \lel Z \rir$. Using $(1)$, we get
\begin{align*}
I_p(V(\cJ_0)) &= \cJ_0,\\
\implies \overline{I_p(V(\cJ_0))}^{\tau} &= \overline{\cJ_0}^{\tau} = \cJ.
\end{align*}
Note that since $V(\cJ_0) = V(\cJ)$, we get $I_p(V(\cJ_0)) = \cI_0$. Proposition \ref{prop_homogeneous_ideal_closure_property} then gives us
\[
I(V(\cJ)) = \overline{\cI_0}^{\tau} = \overline{I_p(V(\cJ_0))}^{\tau} = \cJ,
\]
which completes the proof.
\end{proof}

\begin{remark}

All questions treated in this section have natural counterparts in the setting of infinitely many variables. 
In particular, one might wonder whether a similar Nullstellensatz holds if $d = \infty$. 
For the case of the nc row unit ball, it was shown in \cite[Proposition 3.8]{HarSha23} that there exists a closed homogeneous ideal $\cJ \subset A(\fB_\infty)$ such that $I(V(\cJ)) \neq \cJ$.

The proof of this result does not apply directly to the case of general infinite dimensional operator balls, and in particular to $\fD_\infty$, as it uses the infinite variable nc Drury-Arveson space, but it can be modified somewhat to work in other cases.

\end{remark}

\begin{example}
Let $\cJ$ be the closed ideal in $A(\fD_\infty)$ generated by the monomials $Z_k Z_j$ for all $k, j \in \bN$ with $k \neq j$, and the polynomials $2 Z_{n+1}^2 - Z_n^2$ for all $n \in \bN$. 
Then $\cJ$ is a homogeneous ideal. 
Moreover, by \cite[Proposition 3.8]{HarSha23}, for all $n$, one cannot approximate $Z_n^2$ by linear combinations of the generators of $\cJ$ in the sup norm in $\fB_d$. 
But since $\|\cdot\|_{\fB_\infty} \leq \|\cdot\|_{\fD_\infty}$, we obtain that $Z_n^2 \notin \cJ$ for all $n$. Now, as in \cite[Proposition 3.8]{HarSha23}, if $T = (T_1,T_2, \ldots ) \in V(\cJ)$, then by finite dimensionality there exists some $k \in \bN$ such that
\[
T_k = \sum_{j \neq k} \lambda_j T_j,
\]
where all but finitely many $\lambda_j \in \bC$ are $0$. 
Multiplying by $T_k$ then gives
\[
T_k^2 = \sum_{j\neq k} \lambda_j T_k T_j = 0,
\]
because $T$ annihilates $Z_k Z_j \in \cJ$ for all $j \neq k$. 
But then $T_n^2 = 0$ for all $n$ because $T$ is assumed to annihilate $2 Z_{n+1}^2 - Z_n^2 \in \cJ$.
It follows that $Z_n^2 \in \cI(V(\cJ))$ for all $n$, so $\cJ \subsetneq \cI(V(\cJ))$. 
\end{example}

We also have a trivial version of the Nullstellensatz that is required in later sections.

\begin{proposition}[Trivial Nullstellensatz]\label{prop:trivial nullstellensatz}
If $S \subseteq \bD_Q$ is any set, then 
\[
I(V(I(S))) = I(S). 
\]
Likewise, for every subvariety $\fV \subseteq \bD_Q$, 
\[
V(I(\fV)) = \fV. 
\]
\end{proposition}

\begin{proof}
Tautologically, $I(V(I(S))) \supseteq I(S)$. 
On the other hand, the operations $I(\cdot)$ and $V(\cdot)$ are order reversing. 
Thus, applying $I(\cdot)$ to the containment $V(I(S)) \supseteq S$, we obtain $I(V(I(S))) \subseteq I(S)$. 
This gives the desired equality. 
The other assertion follows similarly. 
\end{proof}

\subsection{The $\tau$-closed algebras $\cA(\fV)$} \label{subsection_algebras_A(V)_over_homogeneous_nc_varieties}

Let $\fV \subset \bD_Q$ be a subvariety in $\bD_Q$. For convenience of notation, let $\cA(\fV)$ denote the algebra $A(\fV)$ if $\cA = A(\bD_Q)$, and $H^{\infty}(\fV)$ if $\cA = H^{\infty}(\bD_Q)$. One would like to know whether $\cA(\fV)$ is related to $\cA$ and if so, what that relation is.

In this Section we show that $\cA(\fV)$ can be realized as a quotient of $\cA$, just as in the case of $A(\fB_d)$ (see \cite[Proposition 9.7]{SSS18}). First, we mention an immediate consequence of Theorem \ref{theorem_BMV_interpolation_full_sets} that will be necessary for the main result in this section.

\begin{proposition} \label{prop_extension_to_H_infty_from_variety}

For every $f \in M_n(\cA(\fV))$, there exists $F \in M_n(H^{\infty}(\bD_Q))$ such that
\[
f = F \rvert_{\fV} \text{ and } \|F\| = \|f\|.
\]
\end{proposition}

\begin{proof}
	
Follows from Theorem \ref{theorem_BMV_interpolation_full_sets} and the fact that varieties are a relatively full nc subset.	
\end{proof}

\begin{theorem} \label{thm_A(V)_is_quotient_of_A(D_d)}

Let $\fV \subset \bD_Q$ be a variety. Then, $R : \cA \rightarrow \cA(\fV)$ defined as
\[
R(F) = F\rvert_\fV, \,  \forall F \in \cA
\]
is a complete contraction with $\ker(R) = I(\fV)$. 
Also, the induced map $\overline{R} : \cA/I(\fV) \rightarrow \cA(\fV)$ is a completely isometric isomorphism of algebras in either one of the following two cases:
\begin{enumerate}
\item $\cA = H^\infty(\bD_Q)$. 
\item $\fV$ is homogeneous. 
\end{enumerate}

\end{theorem}

\begin{proof}
Let $R$ be as in the hypothesis. It is clear that $R$ is a well-defined completely contractive linear map, with $\ker(R) = I(\fV)$ (the fact that $F \vert_\fV \in A(\fV)$ when $F \in A(\bD_Q)$ follows from Proposition \ref{prop_hom_expansion_cesaro_convergent}). 
As this implies $\overline{R}$ is injective, it only remains to show that the quotient map $\overline{R}$ is surjective, and that it is completely isometric.

We achieve this by establishing that $\overline{R}^{-1}$ exists and is completely isometric on $\cA(\fV)$, which amounts to showing that every function in $\cA(\fV)$ has an extension in $\cA$ with (almost) the same norm. 
We shall use Proposition \ref{prop_extension_to_H_infty_from_variety} for this, but note that it is only valid for $H^{\infty}(\bD_Q)$. We must therefore consider the cases $\cA = A(\bD_Q)$ and $\cA = H^{\infty}(\bD_Q)$ separately.

\vspace{2mm}

\textbf{Case 1:} $\cA = H^{\infty}(\bD_Q)$.

\vspace{2mm}

\noindent In this case, the inverse map is defined directly via the norm-preserving extension from Proposition \ref{prop_extension_to_H_infty_from_variety}, and is clearly a completely isometric map. To be precise, given $f \in M_n(\cA(\fV))$, if $F \in M_n(H^{\infty}(\bD_Q))$ is such that
\[
f = F \rvert_{\fV} \text{ and } \|F\| = \|f\|,
\]
then the image $\ol{F} \in M_n(\cA/I(\fV))$ in the quotient satisfies
\[
\|\ol{F}\| = \|f\| \text{ and } \ol{R}(\ol{F}) = f.
\]

\vspace{2mm}

\textbf{Case 2:} $\cA = A(\bD_Q)$ and $\fV$ is homogeneous.

\vspace{2mm}

\noindent In order to prove that $\overline{R}$ is completely isometric, it suffices to show that the inverse map $\overline{R}^{-1}$ is a well-defined completely contractive map on the dense subalgebra $\bC \langle Z \rangle \vert_{\fV}$ of $\cA(\fV)$. Note that the injectivity of $\overline{R}$ implies that \[
\overline{R}^{-1}\left (p \vert_{\fV} \right ) = \overline{p}
\] is well-defined on $\bC \langle Z \rangle \vert_{\fV}$. All we need to show is that this map is completely contractive.

One may wish to proceed as in \textbf{Case 1}, and use Proposition \ref{prop_extension_to_H_infty_from_variety} to obtain $F \in H^\infty(\bD_Q)$ such that $F \vert_\fV = p \vert_\fV$ and $\|F\| = \|p \vert_\fV\|$. However, unless $F \in A(\bD_Q)$, it is unclear how to proceed with the argument. We fix this by replacing $p \vert_\fV$ with a better representative.

To this end, let $p \vert_\fV \in \bC \langle Z \rangle \vert_\fV$ with $p(Z) = \sum_\alpha c_\alpha Z^\alpha$, where all but finitely many $c_\alpha$ are $0$. Note that for every $0 < t < 1$, we have
\begin{align*}
\sup_{X \in \fV} \left \| p\left(\frac{X}{t}\right) - p(X) \right \| &= \sup_{X \in \fV} \left \| \sum_{\alpha \in \bF_d^+ \setminus \{0\}} c_\alpha \left( \frac{1}{t^{|\alpha|}} - 1 \right) X^\alpha \right \| \\
&\leq C \sum_{\alpha \in \bF_d^+ \setminus \{0\}} |c_\alpha| \left( \frac{1}{t^{|\alpha|}} - 1 \right),
\end{align*}
where $C = \sup_{|c_\alpha| \neq 0} \| Z^\alpha \vert_\fV \|$. For $0 < t < 1$, let us define $q^{(t)} \in \bC \langle Z \rangle$ as \[
q^{(t)}(X) = p\left( \frac{X}{t} \right), \, \forall X \in \bM^d.
\] The above observation gives us \[
\lim_{t \rightarrow 1^-} \left \| q^{(t)} \big \vert_\fV - p \vert_\fV \right \| = 0.
\] Therefore, given any arbitrary $\epsilon > 0$, we may choose $0 < t < 1$ such that \[
\left \| q^{(t)} \big \vert_\fV \right \| < (1 + \epsilon) \| p \vert_\fV \|.
\]

Let $q := q^{(t)}$ and note that $q_t = p $. 
By Proposition \ref{prop_extension_to_H_infty_from_variety}, there exists $G \in H^\infty(\bD_Q)$ such that $G \vert_\fV = q \vert_\fV$ and $\|G\| = \| q \vert_\fV \|$. Note that $G$ may or may not be equal to $q$. Let $G = \sum_k G_k$ be the homogeneous expansion of $G$, and note that the homogeneous expansion of $G_t$ is given by the norm convergent series $\sum_k t^k G_k$ (by Proposition \ref{prop_hom_expansion_cesaro_convergent}). In particular, this shows that $G_t \in A(\bD_Q)$. By the homogeneity of $\fV$ and the fact that $G \vert_\fV = q \vert_\fV$, it follows that \[
G_t(X) = G(tX) = q(tX) = q_t (X), \, \forall X \in \fV.
\] Therefore, $\overline{G_t} = \overline{p}$ since \[
G_t \vert_\fV = q_t \vert_\fV = p \rvert_\fV.
\] This gives us \[
\| \overline{p} \| \leq \| G_t \| \leq \| G \| = \left \| q \vert_\fV \right \| < (1 + \epsilon) \left \| p \vert_\fV \right \|.
\] As $\epsilon > 0$ was arbitrarily chosen, we get that $\|\overline{p}\| \leq \| p \vert_\fV \|$ and thus, $\overline{R}^{-1}\vert_{\bC \langle Z \rangle \vert_\fV}$ is contractive.

The above argument can clearly be repeated with matrix-valued polynomials and thus, $\overline{R}^{-1} \vert_{\bC \langle Z \rangle \vert_\fV}$ is completely contractive. This completes the proof. \qedhere

\end{proof}

\begin{remark}

The above proof does not show that every $f \in A(\fV)$ can be extended to some $F \in A(\bD_Q)$ with the same norm. Rather, the extension operator (modulo $I(\fV)$) has the norm-preserving property when $\fV$ is homogeneous.
We leave open the question, for what kind of varieties is it true that every $f \in A(\fV)$ has a norm-preserving extension $F \in A(\bD_Q)$? Also, when can we guarantee that a polynomial $p\vert_{\fV} \in \bC \langle Z \rangle \vert_\fV$ has a norm-preserving extension to $\bD_Q$ via some polynomial $P \in \bC \langle Z \rangle$?

\end{remark}

Theorem \ref{thm_A(V)_is_quotient_of_A(D_d)} combined with Theorem \ref{thm_nullstellensatz_for_homogenous_poyhedra} gives us the following immediate corollary.

\begin{corollary} \label{corollary_isomorphism_subalgebras_of_A}

Let $\cJ \triangleleft \cA$ be a $\tau$-closed homogeneous ideal. The map $\overline{R} : \cA/\cJ \rightarrow \cA(V(\cJ))$ given by
\[
\overline{R}\left(\overline{F}\right) = F\rvert_{V(\cJ)}, \, \forall \overline{F} \in \cA/\cJ
\]
is then a completely isometric isomorphism of algebras.

\end{corollary}

Theorem \ref{thm_A(V)_is_quotient_of_A(D_d)} and Corollary \ref{corollary_isomorphism_subalgebras_of_A} combine to show that there is a one-to-one correspondence between $\tau$-closed algebras over homogeneous nc varieties $\fV \subset \bD_Q$, and quotients of $\cA$ with $\tau$-closed homogeneous ideals $\cJ \triangleleft \cA$: \[
\cA(\fV) \longleftrightarrow \cA/\cJ.
\]

Lastly, we mention an interesting byproduct of Proposition \ref{prop_extension_to_H_infty_from_variety} and Theorem \ref{thm_A(V)_is_quotient_of_A(D_d)}, which will be useful in Section \ref{subsection_unifrom_continuity}.

\begin{corollary} \label{corollary_f_r_on_a_homogeneous_subvariety]}
Let $\fV \subseteq \bD_Q$ be a homogeneous subvariety, and let $f \in H^\infty(\fV)$. For any $r < 1$, define $f_r \in H^\infty(\fV)$ to be the nc function 
\[
f_r(X) = f(rX), \, \forall X \in \fV.
\] 
Then, $f_r \in A(\fV)$ and $\|f_r\| \leq \| f \|$ for each $r < 1$.
\end{corollary}

\begin{proof}
	
Let $\fV \subseteq \bD_Q$ and $f \in H^\infty(\fV)$ be as in the hypothesis. By Proposition \ref{prop_extension_to_H_infty_from_variety}, there exists $F \in H^\infty(\bD_Q)$ with
\[
F \vert_\fV = f \text{ and } \|F\| = \|f\|.
\]
Let $F = \sum_k F_k$ be the homogeneous expansion of $F$. Fix $r < 1$, and define $f_r$ as above. Since $\fV$ is homogeneous and $F\vert_\fV = f$, we get \[
F_r(X) = F(rX) = f(rX) = f_r(X), \, \forall X \in \fV.
\] Therefore, $F_r \vert_\fV = f_r$. Furthermore, since the homogeneous expansion of $F_r$ is given by the norm convergent series $\sum_k r^k F_k$, it follows that $F_r \in A(\bD_Q)$.
	
Now, let $R : A(\bD_Q) \to A(\fV)$ be the restriction map as in Theorem \ref{thm_A(V)_is_quotient_of_A(D_d)} and note that \[
R(F_r) = F_r \vert_\fV = f_r.
\] Since $R$ is a complete contraction, we get \[
\|f_r\| \leq \|F_r\| \leq \|F\| = \|f\|.
\] Furthermore, the norm convergent homogeneous expansion of $F_r$ under the map $R$ provides us with a norm convergent series $\sum_k r^k F_k \vert_\fV$ which converges to $f_r$. Thus, $f_r \in A(\fV)$. \qedhere
	
\end{proof}

\subsection{\bf Uniform continuity} \label{subsection_unifrom_continuity}

We end this section by establishing an alternate description of functions in $A(\fV)$ when $\fV \subseteq \bD_Q$ is a homogeneous subvariety. In particular, we show that functions in $A(\fV)$ are precisely those that extend to $\overline{\fV}$ uniformly. Let us start by detailing the notion of \emph{uniform continuity} for nc functions.

Given $d, l \in \bN$ and a nc set $\fX \subseteq \bM^d$, a nc function $F : \fX \rightarrow \bM^l$ is said to be \emph{uniformly continuous} if for each $\epsilon > 0$, there exists $\delta > 0$ for which 
\[
\|F(X) - F(Y)\|_\infty < \epsilon, \, \forall \, X, Y \in \fX(n), n \in \bN \text{ s.t. } \|X - Y\|_{\infty} < \delta.
\]

Let $U(\bD_Q)$ denote the set of bounded uniformly continuous nc functions on ${\bD_Q}$ i.e., 
\[
U(\bD_Q) = \{ F \in H^\infty(\bD_Q) \, : \, F \text{ is uniformly continuous on } {\bD_Q} \}.
\] 
Since $Q$ is linear and injective, it follows that $\bD_Q$ is a bounded domain. It is then easy to verify that $U(\bD_Q)$ is a norm-closed subalgebra of $H^\infty(\bD_Q)$ that contains the free nc polynomials (see \cite[Lemma 9.1 and Corollary 9.2]{SSS18} for more details). This shows that $A(\bD_Q) \subseteq U(\bD_Q)$. 
Every uniformly continuous function on $\bD_Q$ extends to a uniformly continuous function on its closure $\ol{\bD_Q}$ and therefore, we can also think of $U(\bD_Q)$ as an algebra of functions on $\ol{\bD_Q}$ if we find it to be useful. 

Let us similarly define $U(\fV)$ for a subvariety $\fV \subseteq \bD_Q$ to be the set of $H^\infty(\fV)$ functions that are uniformly continuous on $\fV$. As before, it is clear that $U(\fV)$ is a norm-closed subalgebra of $H^\infty(\fV)$ which contains the polynomials (restricted to $\fV$). We therefore have the natural inclusion $A(\fV) \subseteq U(\fV)$.

Our goal is to show that not only does the reverse inclusion hold for the above, but we may relax the uniform continuity assumption to something weaker, namely \emph{radial uniform continuity}. With notation as before, a nc function $F : \fX \rightarrow \bM^l$ is said to be {\em radially uniformly continuous} if for each $\epsilon > 0$, there exists $\delta > 0$ for which \[
\| F(rX) - F(sX) \|_\infty < \epsilon, \, \forall \, X \in \fX(n), \, n \in \bN \text{ and } r,s \in (0,1) \text{ s.t. } |r -s| < \delta.
\]

Clearly, given a bounded homogeneous nc set $\Omega$, if $F$ is uniformly continuous on $\fX$ then it is also radially uniformly continuous on $\fX$. Moreover, every radially uniformly continuous function on $\fX$ extends to a radially uniformly continuous function up to $\ol{\fX}$. We are now ready to prove the reverse inclusion mentioned earlier.

\begin{proposition} \label{proposition_uniform_continuity_of_A(V)_functions}
Let $\fV \subseteq \bD_Q$ be a homogeneous subvariety. Then, the following are equivalent for $f \in H^\infty(\fV)$:
\begin{enumerate}
\item[$(1)$] $f$ is radially uniformly continuous.
\item[$(2)$] $f \in A(\fV)$.
\item[$(3)$] $f \in U(\fV)$.
\end{enumerate}
\end{proposition}

\begin{proof}
By the discussion above, it is clear that $(2) \Rightarrow (3) \Rightarrow (1)$. Thus we only need to show that $(1) \Rightarrow (2)$ to complete the proof. To this end, let $f \in H^\infty(\fV)$ be radially uniformly continuous. By Corollary \ref{corollary_f_r_on_a_homogeneous_subvariety]}, we know that $f_r \in A(\fV)$ for each $r < 1$. Therefore it suffices to show that $f_r \rightarrow f$ in the sup-norm as $r \rightarrow 1$. Equivalently, we shall show that for each $\epsilon > 0$ there exists $\delta > 0$ such that $\| f_r - f \| < \epsilon$ whenever $r \in (1-\delta, 1)$.
	
First, let us fix an arbitrary $\epsilon > 0$, and let $\delta > 0$ be as in the definition of radial uniform continuity for $f$ that corresponds to $\epsilon/2$. Note that for any given (but fixed) $r \in (1-\delta, 1)$, there exists $X \in \fV$ such that \[
\| f_r - f \| \leq \|f(rX) - f(X)\| + \epsilon/2.
\] Let $0 < s < \delta$ be small enough so that $Y := (1 + s) X \in \fV$ and note that $1/(1+s) \in (1- \delta, 1)$. It is worth pointing out that such a $Y$ exists because $\fV$ is homogeneous. Also, \[
f(rX) = f\left(\frac{r}{1+s} Y\right) \text{ and } f(X) = f\left( \frac{1}{1+s} Y\right).
\] Using the fact that \[
\left\lvert \frac{r}{1+s} - \frac{1}{1+s} \right\rvert < \delta,
\] and that $\delta > 0$ was chosen to be as in the definition of radial uniform continuity for $f$ corresponding to $\epsilon/2$, we obtain for any given $r \in (1 - \delta, 1)$ that \[
\| f_r - f \| \leq \left\| f\left(\frac{r}{1+s} Y\right) - f\left( \frac{1}{1+s} Y\right) \right\| + \epsilon/2 < \epsilon.
\] Thus, $f \in A(\fV)$ as required, and the proof is complete.
\end{proof}


\section{\bf Completely Contractive Representations} \label{section_representations}

Let $Q : \Omega \to \cL(\cR,\cS)_{nc}$ be a linear operator-valued nc map $Q(Z) = \sum_{j = 1}^d Q_jZ_j$ with linearly independent coefficients $Q_j \in \cL(\cR,\cS)$, and let $\bD_Q$ be the corresponding nc operator ball. For $(\cA, \tau)$ as in Section \ref{subsection_varieties_and_homogeneous_ideals}, let us write $\Rep_n(\cA)$ for the space of bounded representations of $\cA$ on $\bC^n$, and $\Rep_n^{cc}(\cA)$ for the space of completely contractive representations of $\cA$ on $\bC^n$. Note that for $\Phi \in \Rep^{cc}_n(\cA)$,
\[
\pi(\Phi) := \Phi(Z) = (\Phi(Z_1), \ldots, \Phi(Z_d))\in M_n^d
\]
is an image of $Z$ under a completely contractive representation.
Note that 
\begin{equation} \label{equation_representations_lie_in_closure}
\| Q(\Phi(Z)) \| =  \left\| \sum_{j = 1}^d Q_j \otimes \Phi(Z_j) \right \| = \left \| \left [I_{\cL(\cR,\cS)} \otimes \Phi\right ]\sum_{j = 1}^d (Q_j \otimes Z_j) \right \|.
\end{equation}
Since $\Phi$ is completely contractive, by \cite[Proposition 2.1.1]{PisierBook} we know that
\begin{equation} \label{equation_pisier_proposition}
\left \| I_{\cL(\cR,\cS)} \otimes \Phi  \right \| \leq \| \Phi \|_{cb} \leq 1.
\end{equation}
Combining (\ref{equation_representations_lie_in_closure}) and (\ref{equation_pisier_proposition}) we get 
\[
\| Q(\Phi(Z)) \| \leq \left \| I_{\cL(\cR,\cS)} \otimes \Phi \right \| \sup_{X \in \bD_Q} \| Q(X) \| \leq 1,
\]
and therefore, it follows that $\pi (\Phi) \in \ol{\bD_Q(n)}$. On the other hand, every $X \in \ol{\bD_Q(n)}$ gives rise to an evaluation representation $\Phi_X$ of $A(\bD_Q)$ on $\bC^n$ given by 
\begin{equation} \label{eq:evaluation functional}
\Phi_X : F \mapsto F(X), 
\end{equation} and clearly $\pi(\Phi_X) = X$. 
Since every continuous representation of $A(\bD_Q)$ is uniquely determined by its values on the generators, we conclude that $\pi \colon \Rep^{cc}_n(A(\bD_Q)) \to \ol{\bD_Q(n)}$ is bijective, whence it is a homeomorphism. The projection map $\pi$ extends to a nc map between the nc sets
\[
\Rep^{cc}(A(\bD_Q)) = \sqcup_{n=1}^\infty \Rep^{cc}_n(A(\bD_Q))
\]
and
\[
\ol{
	\bD_Q} = \sqcup_{n=1}^\infty \ol{\bD_Q(n)}.
\]

Similarly, every $X \in {\bD_Q}$ gives rise to a bounded-pointwise continuous evaluation representation $\Phi_X$ on $H^\infty(\bD_Q)$ given by 
\eqref{eq:evaluation functional} and satisfying $\pi(\Phi_X) = X$. 
On $\Rep^{cc}_n(H^\infty(\bD_Q))$ there is a natural pointwise topology. 
Since $\Rep^{cc}_n(H^\infty(\bD_Q))$ is compact in this topology, the image of $\pi \colon \Rep^{cc}_n(H^\infty(\bD_Q)) \to M_n^d$ is equal to $\ol{\bD_Q(n)}$. 
We summarize what we obtained in the following result. 

\begin{proposition}\label{prop:repspaceDQ}
There exists a continuous surjection $\pi : \Rep^{cc}(\cA, \tau) \to \ol{\bD_Q}$ determined by $\pi(\Phi) = \Phi(Z)$. 
If $\cA = A(\bD_Q)$ then $\pi$ is a homeomorphism. 
If $\cA = H^\infty(\bD_Q)$ then the restriction of $\pi$ to the bounded-pointwise continuous representations in $\pi^{-1}(\bD_Q)$ is a homeomorphism onto $\bD_Q$. 
\end{proposition}

\begin{remark} \label{remark_boundary_representations_for_H_infty}
We are not claiming that the bounded-pointwise continuous representations are precisely those that lie over the open ball $\bD_Q$. 
For example, if $X$ is a jointly nilpotent tuple in the boundary $\partial \bD_Q$, then Proposition \ref{prop_hom_expansion_cesaro_convergent} (1) tells us that $X$ gives rise to a bounded-pointwise continuous representation.

When $\bD_Q = \fB_d$, the bounded-pointwise continuous representations lying over $\partial \fB_d$ are precisely the \emph{pure tuples} (see \cite[Remark 6.2]{SSS18}). 
For $X \in \fB_d$, we know that $\Phi_X$ is the unique element in $\pi^{-1}(X)$ \cite[Theorem 3.2]{DavPitts98b}. 
It is interesting to understand the picture for general $\bD_Q$, however we do not have that answer as of yet.
\end{remark}

%
%

We now turn to study the representations of $\cA(\fV)$ where $\fV \subseteq \bD_Q$ is a subvariety. 

\begin{proposition}\label{prop:rep_cont}
For a homogeneous subvariety $\fV \subseteq \bD_Q$, $\pi : \Rep^{cc}(A(\fV)) \to \ol{\fV}$ given by $\pi(\Phi) = \Phi(Z)$ is a homeomorphism of $\Rep^{cc}_n(A(\fV))$ onto $\ol{\fV(n)}$, $n \in \bN$, with inverse $X \mapsto \Phi_X$. 
\end{proposition}
\begin{proof}
By Theorem \ref{thm_A(V)_is_quotient_of_A(D_d)}, we know that
\[
A(\fV) \cong A(\bD_Q)/ I(\fV).
\]
Every $\Phi \in \Rep^{cc}(A(\fV))$ lifts to a representation of $A(\bD_Q)$ that annihilates $I(\fV)$, and therefore, by the trivial Nullstellensatz (Proposition \ref{prop:trivial nullstellensatz}) corresponds to a unique $X \in \ol{\fV}$.

Conversely, clearly every $X \in \ol{\fV}$ gives rise to an evaluation representation $\Phi_X$ on $A(\fV)$. 
\end{proof}

For $H^\infty(\fV)$ we have the following analogue. 
We do not require that $\fV$ be homogeneous. 

\begin{proposition}\label{prop:bpcont_reps}
If $\fV \subseteq \bD_Q$ is a subvariety, then $\pi : \Rep^{cc}(H^\infty(\fV)) \to \ol{\bD_Q}$ given by $\pi(\Phi) = \Phi(Z)$ is a continuous surjection that restricts to a homeomorphism between the bounded-pointwise continuous $n$-dimensional representations in $\pi^{-1}(\bD_Q)$ and $\fV(n)$. 
\end{proposition}
\begin{proof}
By Theorem \ref{thm_A(V)_is_quotient_of_A(D_d)} again, we know that
\[
H^\infty(\fV) \cong H^\infty(\bD_Q) / I(\fV).
\]
Every representation of $H^\infty(\fV)$ lifts to a representation of $H^\infty(\bD_Q)$ that annihilates $I(\fV)$, and therefore the image of $\pi$ is contained in $\ol{\bD_Q}$. 

Now, every $X \in \fV$ clearly defines a bounded-pointwise continuous point evaluation representation $\Phi_X$.
Conversely, if $X = \pi(\Phi) \in \bD_Q$ is the projection of a bounded-pointwise continuous representation then $\Phi$ lifts to a bounded-pointwise continuous representation of $H^\infty(\bD_Q)$, therefore this lift is an evaluation map $\Phi_X$ that annihilates $I(\fV)$, so $X \in \fV$. 
Thus $\pi$ restricts to a bijection from the bounded-pointwise continuous representations in $\pi^{-1}(\bD_Q)$ and $\fV$. 
One easily verifies that this bijection is a homeomorphism.
\end{proof}

As pointed out in Remark \ref{remark_boundary_representations_for_H_infty}, we do not understand the complete picture of the representation theory yet and hence, certain aspects of the isomorphism problem for $H^\infty(\fV)$ are still hidden. We therefore turn our attention to solving the isomorphism problem for the algebras $A(\fV)$ when $\fV$ is a homogeneous subvariety in some nc operator ball $\bD_Q$.

\section{\bf Classification of the Algebras $A(\fV)$} \label{section_classification_of_A(V)}

If $\varphi \colon A(\fV_1) \to A(\fV_2)$ is a bounded isomorphism, then we obtain a natural nc map 
\begin{align*}
\varphi^* \colon \Rep(A(\fV_2)) &\rightarrow \Rep(A(\fV_1)) \\
\Phi &\mapsto \Phi\circ \varphi. 
\end{align*} 
If $\varphi$ is completely contractive then $\varphi^*$ restricts to a map from $\Rep^{cc}(A(\fV_2))$ to $\Rep^{cc}(A(\fV_1))$. 
The induced map $\varphi^*$ is clearly a nc map. 
Our goal is to further understand this map and use this to classify the algebras $A(\fV)$ up to completely isometric isomorphism. 

\subsection{The basic classification theorem}\label{subection_basic_classification}

For $i = 1,2$, let $Q_i$ and $\bD_{Q_i}$ be as in (\ref{eq:operator_ball}), and let $\fX_i \subseteq \bD_{Q_i}$ be nc sets.
If a nc map $F \colon \fX_1 \to \fX_2$ is bijective, then its inverse $F^{-1} \colon \fX_2 \to \fX_1$ is also a nc map. 
In this situation, we say that $F$ is a {\em nc biholomorphism} between $\fX_1$ and $\fX_2$, or that $\fX_1$ and $\fX_2$ are {\em nc biholomorphic}.

\begin{theorem}\label{thm:iso_cont}
Let $\fV_i \subseteq \bD_{Q_i}$ be two homogeneous subvarieties of nc operator balls. 
Then, for every completely contractive homomorphism $\varphi \colon A(\fV_1) \to A(\fV_2)$ there exists a uniformly continuous nc map $F \colon \ol{\fV_2} \to \ol{\fV_1}$ such that 
\[
\varphi(f) = f\circ F , \, \forall f \in A(\fV_1) .
\]
Conversely, every uniformly continuous nc map $F \colon \ol{\fV_2} \to \ol{\fV_1}$ gives rise to a continuous completely contractive homomorphism $\varphi : A(\fV_1) \to A(\fV_2)$.

Moreover, the following are equivalent:

\begin{enumerate}

\item There is a completely isometric isomorphism $\varphi : A(\fV_1) \to A(\fV_2)$.

\item There is a uniformly continuous nc biholomorphism $G : \overline{\fV_1} \rightarrow \overline{\fV_2}$ with a uniformly continuous inverse.

\item There is a uniformly continuous nc biholomorphism $G : \fV_1 \to \fV_2$ with a uniformly continuous inverse.

\end{enumerate}

\end{theorem}
\begin{proof}
Using Proposition \ref{prop:rep_cont}, identify $\Rep^{cc}(A(\fV_i))$ with $\ol{\fV_i}$. 
Then $F := \varphi^*$ is a nc map. 
Under the identification $\Rep^{cc}(A(\fV_i)) = \ol{\fV_i}$, for every $X \in \fV_2$ and every $1 \leq i \leq d$, 
\[
F_i(X) = \varphi^*(\Phi_X)(Z_i) = \Phi_X\circ \varphi(Z_i) =  \varphi(Z_i)(X).
\]
In particular,
\[
F_i = \varphi(Z_i), \, \forall \, 1 \leq i \leq d
\]
Since $\varphi(Z_i) \in A(\fV_2)$, it is uniformly continuous. The converse follows from Proposition \ref{proposition_uniform_continuity_of_A(V)_functions}, since every uniformly continuous nc map on $\overline{\fV_2}$ is automatically in $A(\fV_2)$.

Let us now turn to the equivalence of $(1)$, $(2)$ and $(3)$.

\vspace{2mm}

$(1) \Leftrightarrow (2)$ : If $\varphi$ is a completely isometric isomorphism, then the map $F : \ol{\fV_2} \rightarrow \ol{\fV_1}$ is a nc biholomorphism with inverse $G : \ol{\fV_1} \rightarrow \ol{\fV_2}$ that is given by $G := (\varphi^{-1})^*$. $F$ and $G$ are clearly uniformly continuous by the above discussion. The converse is easy to verify.

$(2) \Rightarrow (3)$ : Let $G : \overline{\fV_1} \to \overline{\fV_2}$ be as in $(2)$, and note that it suffices to show that $G(\fV_1) \subseteq \fV_2$. For this, we modify the proof of the boundary value principle (Theorem \ref{theorem_bdy_value_principle}) to work for homogeneous subvarieties instead of a general nc domain $\Omega$.

Suppose there is an $X_0 \in \fV_1 (n)$ for some $n \in \bN$ such that $G(X_0) \in \partial \fV_2$. 
First, we show that $G(0) \in \partial \fV_2$. 
If $X_0 = 0$ then there is nothing to show, so assume $X_0 \neq 0$. 
Since $\fV_1$ is homogeneous, we have
\[
X_0 \in D_{X_0} := \left\{ \lambda \frac{X_0}{\|Q_1(X_0)\|} : \lambda \in \bD \right\} \subset \fV_1.
\]
As in the boundary value principle, let $\Lambda \in (\cL(\cR_2,\cS_2))^*$ be a bounded linear functional such that $\| \Lambda \| = 1$ and $\| \Lambda(Q_2(G(X_0))) \| = 1$. Then, using the maximum modulus principle for
\[
H := \Lambda \circ Q_2 \circ G : D_{X_0} \to \overline{\bD},
\]
we get that $G(D_{X_0}) \subset \partial \fV_2$. 
Thus, $G(0) \in \partial \fV_2$ as claimed.

We can therefore repeat the above argument for every such disk $D_{X}$ given by an arbitrary $X \in \fV_1$, using the fact that $G(0) \in \partial \fV_2$, and show that $G(\fV_1) \subseteq \partial \fV_2$. 
Clearly, this is absurd as it implies that $G(\overline{\fV_1}) = \partial \fV_2$, and so $0 \in \fV_2$, for instance, has no pre-image under $G$. 
Therefore, $G(\fV_1) \subseteq \fV_2$ as required.

$(3) \Rightarrow (2)$ : This implication holds using the fact that uniformly continuous nc functions extended up to the boundary of the nc that set they are defined on (see Section \ref{subsection_unifrom_continuity}).
\end{proof}

\begin{remark} Note that the proof of $(2) \Rightarrow (3)$ above can be modified to provide a boundary value principle for nc maps $F : \fX \to \overline{\bD_Q}$, where $\fX$ is some homogeneous nc set i.e., $\lambda \fX \subseteq \fX$ for all $\lambda \in \bD$. We shall prove a more elegant result for a certain class of nc operator balls in Theorem \ref{thm:iso_cont_1-injective}, wherein the nc biholomorphism $G : \fV_1 \to \fV_2$ as in $(3)$ above is shown to be given by the restriction of a nc biholomorphism $\widetilde{G} : \bD_{Q_1} \to \bD_{Q_2}$.
\end{remark}

\begin{example} \label{example_image_could_lie_in_boundary_variety}
In general, we cannot conclude that a nc map $F$ as in the first part of Theorem \ref{thm:iso_cont} always maps $\fV_2$ into $\fV_1$. For instance, let $\bD_{Q_1} = \bD_{Q_2} = \fD_2$ and consider
\begin{align*}
\fV_1 &= V\left(\langle Z_1 Z_2 - Z_2 Z_1 \rangle\right) \text{ -- the commutative free subvariety in } \fD_2, \\
\fV_2 &= V\left(\langle Z_2 \rangle\right) = \left\{ (X,0) \in \fD_2 \, : \, X \in \fD_1 \right\}.
\end{align*}
Now, define $F : \ol{\fV_2} \rightarrow \ol{\fV_1}$ as in Example \ref{example_image_could_lie_in_boundary}  i.e., \[
F(X,0) = (X,I_n), \, \forall X \in \fD_1(n), \, n \in \bN.
\] Then, we clearly see that $F(\ol{\fV_2}) \subset \partial \fV_1$. On the other hand, when $\bD_{Q_1} = \bD_{Q_2} = \fB_d$, by \cite[Lemma 6.11]{SSS18} it follows that such a nonconstant $F$ must always map $\fV_2$ into $\fV_1$.
\end{example}

For some special nc operator balls, we can combine Theorems \ref{theorem_BMV_interpolation_full_sets} and \ref{theorem_bdy_value_principle} to make an important observation which allows us to say more about the nc biholomorphism between the subvarieties in Theorem \ref{thm:iso_cont}. We explore this in the next section.

\subsection{Injective NC operator balls} \label{subsection_1-injective_nc_operator_balls} Recall that an operator space $\cE$ is said to be {\em injective} if it is an injective object in the category of operator spaces with completely contractive maps, that is, whenever $\cF \subset \cG$ are operator spaces and $u : \cF \to \cE$ is completely contractive, then there exists a completely contractive extension $\tilde{u} : \cG \to \cE$ such that $\tilde{u}\big|_{\cF} = u$ (this definition is from \cite{Ruan89}. Injective operator spaces are also sometimes referred to as \emph{1-injective}; see e.g. \cite[Chapter 24]{PisierBook}).

For arbitrary Hilbert spaces $\cR, \, \cS$ and a finite dimensional operator space $\cE \subseteq \cL(\cR,\cS)$, the following are equivalent \cite{Ruan89, Smith00}:

\begin{enumerate}

\item $\cE$ is injective.

\item Given any completely isometric embedding $\cE \subseteq B(\cH)$ for a Hilbert space $\cH$, there exists a completely contractive projection $\Pi : B(\cH) \to \cE$.

\item There exists a finite dimensional C*-algebra $\cB$ and a projection $\beta \in \cB$ such that $\cE \cong \beta \cB \beta^\perp$ completely isometrically.

\end{enumerate}

Let $Q(Z) = \sum_j Q_j Z_j : \bM^d \to \cL(\cR, \cS)_{nc}$ and $\bD_Q$ be as in (\ref{eq:operator_ball}), and consider the finite dimensional operator space $\cE := \operatorname{span}\{Q_1, \dots, Q_d\} \subseteq \cL(\cR,\cS)$. We say that $\bD_Q$ is \emph{injective} when $\cE$ is an injective operator space.

\begin{example} \label{example_1-injective_nc_operator_balls} $(1)$ It is clear that $\fB_d$ and $\fD_d$ are injective for all $d \in \bN$.

$(2)$ By extension, it follows that for $1 \leq d_1 \leq \dots \leq d_l \leq d$, $\bD_Q \subset \bM^d$ is injective where
\[
Q(Z) = \begin{bmatrix}
\, [Z_1 \dots Z_{d_1}] & & 0 \\
& \ddots & \\
0 & & [Z_{d_l + 1} \dots Z_d] \,
\end{bmatrix}.
\]

$(3)$ Fix $2 \leq l \in \bN$ and let $d = l(l+1)/2$. Consider $\bD_Q \subset \bM^d$ where \[
Q(Z) = \left[\begin{matrix}
Z_{11} & Z_{12} & \dots & Z_{1l} \\
0 & Z_{22} & \dots & Z_{2l} \\
\vdots & \vdots & \ddots & \vdots \\
0 & 0 & \dots & Z_{ll}
\end{matrix}\right].
\] Note that
\[
\cE := \operatorname{span}\{ Q_{jk} : 1 \leq j \leq k \leq l \} = UT_l \subset B(\bC^l)
\]
is the space of $l \times l$ upper-triangular matrices. We claim that $\cE$ is not injective. For the sake of argument, suppose $\cE$ is injective. By the equivalence of injectivity above, we get a completely contractive projection $\Pi : B(\bC^l) \to \cE$. In particular, $\Pi \vert_{\cE} = {\bf id}_{\cE}$ and therefore $\Pi$ is also unital. Since a completely contractive unital map is also completely positive, we get that $\Pi$ is a unital completely positive map. 
In particular, $\Pi$ must be a self-adjoint map, and we conclude the range $\Pi(B(\bC^l)) = \cE$ of $\Pi$ is closed under the adjoint operation. 
This conclusion is absurd, so we see that $\cE$ is not injective. 
%

\end{example}

Let us now see why injectivity is useful in our context. Combining Ball, Marx and Vinnikov's extension theorem (Theorem \ref{theorem_BMV_interpolation_full_sets}) and the boundary value principle (Theorem \ref{theorem_bdy_value_principle}) for injective nc operator balls, we get the following nonlinear generalization of injectivity.

\begin{theorem} \label{thm:extend_to_ball}
For $i = 1,2$, let $Q_i$ and $\bD_{Q_i}$ be as in (\ref{eq:operator_ball}), and let $\mathfrak{X} \subset \bD_{Q_1}$ be a relatively full nc subset. If $\bD_{Q_2}$ is injective, then any nc map $S_0 : \mathfrak{X} \rightarrow \ol{\bD_{Q_2}}$ can be extended to a nc map $S : \bD_{Q_1} \rightarrow \ol{\bD_{Q_2}}$ i.e., $S \vert_{\fX} = S_0$. Furthermore, we have the following dichotomy:
	
\begin{enumerate}
		
\item[$(1)$] If $S_0(\fX) \subseteq \bD_{Q_2}$, then $S(\bD_{Q_1}) \subseteq \bD_{Q_2}$.
		
\item[$(2)$] If $S_0(X) \in \partial \bD_{Q_2}$ for some $X \in \fX$, then $S(\bD_{Q_1}) \subseteq \partial \bD_{Q_2}$.
		
\end{enumerate}

\end{theorem} 

\begin{proof}
Let $Q_2(Z) = \sum_{j=1}^{d_2} Q^{(2)}_j Z_j : \bM^{d_2} \to \cL(\cU,\cV)_{nc}$ for some Hilbert spaces $\cU, \cV$ with
\[
\cE_2 := \operatorname{span}\left\{Q^{(2)}_1, \ldots, Q^{(2)}_{d_2}\right\} \subseteq \cL(\cU,\cV).
\]
If $\bD_{Q_2}$ is injective, then we have a completely contractive projection $\Pi_2 : \cL(\cU, \cV) \to \cE_2$. 
Let $\fX \subseteq \bD_{Q_1}$ and $S_0 : \fX \to \bD_{Q_2}$ be as in the hypothesis. Define $F_0 : \mathfrak{X} \rightarrow \cL(\cU,\cV)_{nc}$ by $F_0 = Q_2 \circ S_0$. 
Then $F_0$ is clearly a nc map such that $\| F_0(X) \| < 1$ for all $X \in \mathfrak{X}$. Now, by Theorem \ref{theorem_BMV_interpolation_full_sets}, there exists an extension of $F_0$ to a nc map $F \colon \bD_{Q_1} \to  \cL(\cU,\cV)_{nc}$ such that \[
\sup_{X \in \bD_{Q_1}} \| F(X) \| = \sup_{X \in \mathfrak{X}} \| F_0(X) \| \leq 1. 
\] If we define $S \colon \bD_{Q_1} \to \bM^{d_2}$ by $S = Q_2^{-1} \circ \Pi_2 \circ F$, then $S$ is a nc function that extends $S_0$, and $S(\bD_{Q_1}) \subseteq \ol{\bD_{Q_2}}$ as required. 
Items $(1)$ and $(2)$ now follow easily from Theorem \ref{theorem_bdy_value_principle}.
\end{proof}

We now reformulate Theorem \ref{thm:iso_cont} to the case where both $\bD_{Q_i}$ are injective.


\begin{theorem} \label{thm:iso_cont_1-injective}
Let $\fV_i \subseteq \bD_{Q_i}$ be two homogeneous subvarieties of injective nc operator balls. Then, $A(\fV_1) \cong A(\fV_2)$ completely isometrically if and only if there is a uniformly continuous nc biholomorphism $G : \fV_1 \to \fV_2$ with a uniformly continuous inverse.

In case such a nc biholomorphism $G : \fV_1 \to \fV_2$ exists, we obtain nc maps $\widetilde{G} : \bD_{Q_1} \to \bD_{Q_2}$ and $\widetilde{F} : \bD_{Q_2} \to \bD_{Q_1}$ such that
\[
\widetilde{G} \vert_{\fV_1} = G \text{ and } \widetilde{F} \vert_{\fV_2} = G^{-1}.
\]
\end{theorem}

\begin{proof}
The first part is the equivalence $(1) \Leftrightarrow (3)$ of Theorem \ref{thm:iso_cont} and the second part then follows from Theorem \ref{thm:extend_to_ball}.
\end{proof}

In the next section, with the help of some function theoretic tools, we establish that in certain cases $\widetilde{G}$ above is a nc biholomorphism with $\widetilde{F}$ as its bijective inverse.

\subsection{NC biholomorphisms between NC operator balls} \label{subsection_biholomorphisms_between_D_Q}

Let us now try to understand the structure of nc biholomorphisms on nc operator balls. 
We shall extensively use the properties of the \emph{first order nc derivative of a nc map} (as studied in \cite[Section 2]{KVV14}).

Given a nc domain $\Omega \subset \bM^l$, let $F = (F_1, \dots, F_d) : \Omega \rightarrow \bM^d$ be a nc holomorphic map. 
For a given $Y \in \Omega(n)$ and $X \in M_n^l$, 
the first order nc derivative of $F$ at $Y$ is defined to be the linear map $\Delta F(Y,Y) : M_n^l \to M_n^d$ determined by 
\[
F\left( \begin{bmatrix} Y & X \\ 0 & Y \end{bmatrix}\right) = \begin{bmatrix} F(Y) & \Delta F(Y,Y)(X) \\ 0 & F(Y) \end{bmatrix}
\]
for all $X \in M_n^l$ with sufficiently small norm. 
The following nc version of Cartan's uniqueness theorem is known (see \cite[Theorem 6.7]{SSS18}), but we state it here for the reader's convenience.

\begin{theorem} \label{theorem_nc_sartan_uniqueness_general}

Let $\Omega \subset \bM^d$ be a uniformly bounded nc domain (i.e. bounded and open w.r.t. the uniform nc topology), and let $F : \Omega \rightarrow \Omega$ be a nc holomorphic function. If there exists $Y \in \Omega$ such that $F(Y) = Y$ and $\Delta F(Y,Y) = I$, then $F(X) = X$ for every $X \in \Omega$.

\end{theorem}

Our goal is to establish, using Theorem \ref{theorem_nc_sartan_uniqueness_general}, that every nc biholomorphism between two subvarieties (possibly in different nc operator balls) arises from a nc biholomorphism of the corresponding nc operator balls. Of course this cannot be true in general (because, for example, $\fD_d \subset \fD_k$ is a subvariety for $d < k$), so we need some additional assumptions on the subvarieties. For this we first need to define \emph{matrix spans} of subsets of $\bM^d$. We shall repeatedly use the fact that
\[
M^d_n \cong \bC^d \otimes M_n, \, \forall n, d \in \bN.
\]

For a given subset $\fX \subseteq \bM^d$, we define its {\em matrix span} to be the graded set \[
\operatorname{mat-span} \fX = \sqcup_{n = 1}^\infty \operatorname{mat-span} \fX(n),
\] 
where 
\[
\operatorname{mat-span} \fX(n) := \spn \left\{ [I_d \otimes T]X \, : \, X \in \fX(n), \, T \in \cL(M_n) \right\}.
\]
The following result, which is known (see \cite[Lemma 8.2]{SSS18}), provides a concrete description of matrix-spans of nc subsets. We state it here for the reader's convenience.

\begin{lemma} \label{lemma_description_of_mat-span}
	
Let $\fX \subset \bM^d, \, d \in \bN$ be a nc set. Then, we have the following:

\begin{enumerate}
	
\item[$(1)$] For all $n \in \bN$, there exists a subspace $V_n \subseteq \bC^d$ such that $\emph{mat-span } \fX(n) = V_n \otimes M_n$.

\item[$(2)$] There is a minimal subspace $V \subseteq \bC^d$ such that $\emph{mat-span } \fX(n) \subseteq V \otimes M_n$ for all $n$.

\item[$(3)$] If $V$ is as above, then for all sufficiently large $n$, we have $\emph{mat-span } \fX(n) = V \otimes M_n$.
		
\end{enumerate}
	
\end{lemma}

In view of the above lemma, we say that a subvariety $\fV \subseteq \bD_Q$ is \emph{matrix-spanning} if \begin{equation} \label{equation_matrix-spanning_subvariety}
\operatorname{mat-span} \fV(n) = \bC^d \otimes M_n, \text{ for some } n \in \bN.
\end{equation}

We now show how matrix-spans are useful in the study of nc holomorphic maps. 
If $F : \fX \to \bM^d$ is nc holomorphic, we shall write $\Delta F(0,0) \rvert_{\fX} = I$ to mean that
\[
\Delta F\left(0^d_n,0^d_n\right) \rvert_{\fX(n)} = {\bf id}_{M^d_n}, \, \forall \, n \in \bN.
\]
Here, $0^d_n$ denotes the $d$-tuple of $n \times n$ zero matrices, and ${\bf id}_{M^d_n}$ denotes the identity operator on the space of $d$-tuples of $n \times n$ matrices. Note that this usage is somewhat different from the $\Delta F(Y,Y) = I$ appearing in Theorem \ref{theorem_nc_sartan_uniqueness_general}, which means that there is some specific $n$ and a specific $Y \in \Omega(n)$ such that $\Delta F(Y,Y) = {\bf id}_{M^d_n}$.

With the above notation in place we have the following useful lemmas.

\begin{lemma} \label{lemma_extending_delta_to_mat-span_if_identity}
Let $\bD_Q \subset \bM^d, \, d \in \bN$ be a nc operator ball, and let $F : \bD_Q \to \bM^d$ be a nc holomorphic map. Then, $\Delta F(0,0)$ acts as $A \otimes {\bf id}_{M_n}$ on $\bC^d \otimes M_n$ for each $n \in \bN$, where $A \in \cL(\bC^d)$ is some linear map.

In particular, if $\Delta F(0,0) \rvert_{\fX} = I$ for some subset $\fX \subset \bM^d$, then $\Delta F(0,0) \rvert_{\operatorname{mat-span} \fX} = I$.
\end{lemma}

\begin{proof}
By the properties of nc derivatives (see \cite[Proposition 2.15 (2X,2Y)]{KVV14}), we have \[
\Delta F(0,0)([I_d \otimes T]X) = [I_d \otimes T] \Delta F(0,0)(X), \, \forall T \in \cL(M_n).
\] Therefore, $\Delta F(0,0)$ commutes with all operators on $\bC^d \otimes M_n$ of the type $I_d \otimes T$. This property, together with the fact that $\Delta F(0,0)$ is a linear nc map, implies that $\Delta F(0,0)$ acts like $A \otimes {\bf id}_{M_n}$ on $\bC^d \otimes M_n$ for each $n \in \bN$, as required.

The second part of the lemma then follows easily from this identification.
\end{proof}

We also use the following property of $\Delta F(0,0) \vert_{\fV}$ for a homogeneous subvariety $\fV$.

\begin{lemma} \label{lemma_identity_on_subvariety}
Let $\fV \subseteq \bD_Q$ be a homogeneous subvariety. If $F : \bD_Q \rightarrow \bD_Q$ is a nc holomorphic map such that $F\rvert_{\fV} = {\bf id}_{\fV}$, then $\Delta F(0,0) \rvert_{\fV} = I$.
\end{lemma}

\begin{proof}
By the \emph{nc difference-differential formula} (see \cite[Theorem 2.10]{KVV14}), we know that for all $t \in \bC \setminus \{0\}$ and $X \in \bD_Q$, 
\[
t \Delta F(tX,0)(X) = \Delta F(tX,0)(tX - 0) = F(tX) - F(0) = F(tX).
\] 
Using the holomorphicity of $\Delta F(0,0)$ (see \cite[Proposition 7.46]{KVV14}) we get that for all $X \in \fV$, \[
\Delta F(0,0) (X) = \lim_{t \to 0} \Delta F(tX,0) (X) = \lim_{t \to 0} \frac{1}{t} F(tX) = X.
\] Therefore, $\Delta F(0,0)$ is identity on $\fV$ as required.
\end{proof}

We now have the following immediate corollary of Lemmas \ref{lemma_description_of_mat-span}, \ref{lemma_extending_delta_to_mat-span_if_identity}, and \ref{lemma_identity_on_subvariety}.

\begin{corollary}\label{cor:fix}
Let $\fV \subseteq \bD_Q \subset \bM^d$ be a matrix-spanning homogeneous subvariety.
If $H : \bD_Q \to \bD_Q$ is a nc holomorphic map such that $H \rvert_\fV = {\bf id}_{\fV}$, then $H = {\bf id}_{\bD_Q}$. 
\end{corollary}
\begin{proof}
Combining Lemmas \ref{lemma_extending_delta_to_mat-span_if_identity} and \ref{lemma_identity_on_subvariety}, we see that $\Delta H(0,0)$ is identity on $\operatorname{mat-span} \fV$. With the notation from Lemmas \ref{lemma_description_of_mat-span} and \ref{lemma_extending_delta_to_mat-span_if_identity}, this means that \[
[A \otimes {\bf id}_{M_n}](X) = \Delta F(0,0) (X) = X, \, \forall X \in V_n \otimes M_n, \, n \in \bN.
\] By the hypothesis, since $V_n = \bC^d$ for some $n \in \bN$, we get that $A = I_d$. In particular, $\Delta F(0,0) = I$, and by Theorem \ref{theorem_nc_sartan_uniqueness_general} we get that $H = {\bf id}_{\bD_Q}$ as required. \qedhere
\end{proof}

The above results provide us with the following description of nc biholomorphic maps on matrix-spanning homogeneous subvarieties of injective nc operator balls.

\begin{theorem} \label{theorem_cartan_uniqueness_for_subvariety}

For $i = 1,2$, let $\fV_i \subseteq \bD_{Q_i} \subset \bM^{d_i}$ be matrix-spanning homogeneous subvarieties of injective nc operator balls. If $G : \fV_1 \rightarrow \fV_2$ is nc biholomorphic, then there exists a nc biholomorphism $\widetilde{G} : \bD_{Q_1} \rightarrow \bD_{Q_2}$ such that $\widetilde{G} \rvert_{\fV_1} = G$.
\end{theorem}

\begin{proof}
Let $G : \fV_1 \to \fV_2$ be a nc biholomorphism with bijective inverse $F : \fV_2 \to \fV_1$. By Theorem \ref{thm:extend_to_ball}, there exist nc maps $\widetilde{G} : \bD_{Q_1} \to \bD_{Q_2}$ and $\widetilde{F} : \bD_{Q_2} \to \bD_{Q_1}$ such that
\[
\widetilde{G} \rvert_{\fV_1} = G \text{ and } \widetilde{F} \rvert_{\fV_2} = F.
\]
Now, let $H := \widetilde{F} \circ \widetilde{G}$ and note that $H \rvert_{\fV_1} = {\bf id}_{\fV_1}$. 
By Corollary \ref{cor:fix}, we find that $H = {\bf id}_{\bD_{Q_1}}$. Arguing similarly, we get $\widetilde{G} \circ \widetilde{F} = {\bf id}_{\bD_{Q_2}}$, and thus $\widetilde{G} = \widetilde{F}^{-1}$ as required.
\end{proof}

One might wonder how rigid the matrix-spanning condition (\ref{equation_matrix-spanning_subvariety}) is, and if at all there are any interesting examples of such subvarieties. We explore this in the next section.

\subsection{Matrix-spanning subvarieties} \label{subsection_matrix-spanning-subvarieties}

As a warm-up for our main result in this section, we start by giving some important examples of homogeneous subvarieties $\fV \subseteq \bD_Q$ for which $\operatorname{mat-span} \fV = \bM^d$. In particular, these subvarieties are matrix-spanning for $\bD_Q$.

\begin{example} \label{example_subvarieties_with_full_matrix_span}

\textbf{$(1)$ Commutative free subvariety:} For a nc operator ball $\bD_Q \subset \bM^d$, the commutative free subvariety $\fC \bD_Q$ is given by \begin{align*}
\fC \bD_Q &= V\left(\langle Z_k Z_l - Z_l Z_k : 1 \leq k,l \leq d \rangle \right) \\
    &= \left\{ (X_1, \dots, X_d) \in \bD_Q : X_k X_l = X_l X_k,\, \forall \, 1 \leq k,l \leq d \right\}.
\end{align*} Let $X = (X_1, \dots, X_d) \in \bM^d_n, \, n \in \bN$ be arbitrary, and note that \begin{equation} \label{equation_mat-span_X=X_1+....X_d}
X = (X_1, 0, \dots, 0) + (0, X_2, \dots, 0) + \dots + (0, 0, \dots, X_d) 
\end{equation} Each $d$-tuple in (\ref{equation_mat-span_X=X_1+....X_d}) lies in $\fC \bD_Q$ (after scaling appropriately), and thus $X \in \operatorname{mat-span} \fC \bD_Q$. Since $X \in \bM^d_n$ was arbitrarily chosen, it then follows that $\operatorname{mat-span} \fC \bD_Q = \bM^d$.

\vspace{2mm}

\textbf{$(2)$ q-commuting free subvariety:} For a nc operator ball $\bD_Q \subset \bM^2$ and some $q \in \bC$, the corresponding \emph{$q$-commuting free subvariety} $\fC^q \bD_Q$ is given by \begin{align*}
\fC^q \bD_Q &= V\left(\langle Z_1 Z_2 - q Z_2 Z_1 \rangle\right) \\
	  &= \left\{ (X_1,X_2) \in \bD_Q : X_1 X_2 = q X_2 X_1 \right\}.
\end{align*} Let $X = (X_1,X_2) \in \bM^2_n$, $n \in \bN$ be arbitrary, and note as in (\ref{equation_mat-span_X=X_1+....X_d}) that \[
X = (X_1,0) + (0,X_2).
\] 
One can show as above that $\operatorname{mat-span} \fC^q \bD_Q = \bM^2$.

For $\bD_Q = \fB_2$, $q$-commuting subvarieties have already been studied in the context of the isomorphism problem. Details on this can be found in \cite[Section 8]{SSS20}.

\end{example}

We now give a concrete description of matrix-spanning subvarieties (cf. \cite[Lemma 3.4]{Sham18}).

\begin{theorem} \label{theorem_mat-spanning_subvarieties}
Let $\fV \subseteq \bD_Q$ be any subvariety of some nc operator ball. Then $\fV$ is matrix-spanning if and only if there are no linear homogeneous free nc polynomials in $I(\fV)$.
\end{theorem}

\begin{proof}
Suppose there is a linear homogeneous free nc polynomial $p \in I(\fV)$. Then, clearly, the proper subspace $V \subsetneq \bC^d$ cut-out by $p \vert_{M^d_1}$ is such that
\[
\fV (n) \subseteq V \otimes M_n, \, \forall \, n \in \bN.
\]
Now if $T \in \cL(M_n)$ for some $n \in \bN$, then note that for every $X \in \fV(n)$ we get
\[
p\left(\left[I_d \otimes T\right] X\right) = p \left(T\left(X_1\right), \dots, T\left(X_d\right)\right) = T(p(X)) = 0.
\]
In other words,
\[
\operatorname{mat-span} \fV(n) \subseteq V \otimes M_n, \, \forall \, n \in \bN.
\]
Therefore, $\fV$ cannot be matrix-spanning for $\bD_Q$.

Conversely, suppose $\fV$ is not matrix-spanning for $\bD_Q$. By Lemma \ref{lemma_description_of_mat-span}, there exists a proper subspace $V \subsetneq \bC^d$ such that
\begin{equation} \label{equation_mat-span_in_proper_subspace}
\operatorname{mat-span} \fV(n) \subseteq V \otimes M_n, \, \forall \, n \in \bN.
\end{equation}
Therefore, we can find some non-zero $f \in \left( \bC^d \right)^*$ such that
\[
f(v) = 0, \, \forall \, v \in V.
\]
We can now define a linear homogeneous free nc polynomial $p \in \bC \langle Z \rangle$ by
\[
p(X) := \left[ f \otimes {\bf id}_{M_n} \right](X), \, \forall \, X \in M_n^d, \, n \in \bN.
\]
For a given $n \in \bN$ and $X \in \fV(n)$, by (\ref{equation_mat-span_in_proper_subspace}), we know that there exist $v_j \in V$ and $X_j \in M_n$, $1 \leq j \leq d$ such that
\[
X = \sum_{j = 1}^d v_j \otimes X_j.
\]
Note that
\[
p(X) = \left[ f \otimes {\bf id}_{M_n} \right] \left( \sum_{j = 1}^d v_j \otimes X_j \right) = \sum_{j = 1}^d f(v_j) X_j = 0.
\]
As $n \in \bN$ and $X \in \fV(n)$ were arbitrarily chosen, we obtain a linear homogeneous free nc polynomial $p \in I(\fV)$.
\end{proof}

We now show that when $\fV \subset \bD_Q$ is not matrix-spanning, we can find a `smaller' nc operator ball $\bD_P \subset \bD_Q$ such that $\fV \subseteq \bD_P$ is matrix-spanning.

\begin{theorem} \label{theorem_existence_of_sub_nc_operator_ball_with_full_matrix_span_for_V}
Let $\fV \subset \bD_Q \subset \bM^d$ be a subvariety in some nc operator ball. If $\fV$ is not matrix-spanning for $\bD_Q$, then there exists a nc operator ball $\bD_P \subset \bM^e$ for some minimal $e < d$ such that, under the appropriate identifications, $\fV \subseteq \bD_P \subset \bD_Q$, and for which $\fV$ is matrix-spanning.
\end{theorem}

\begin{proof}
Let $\fV \subset \bD_Q$ be as in the hypothesis. Since $\fV$ is not matrix-spanning for $\bD_Q$, we use Lemma \ref{lemma_description_of_mat-span} to find a proper subspace $V \subsetneq \bC^d$ such that
\[
\operatorname{mat-span} \fV(n) \subseteq V \otimes M_n, \, \forall \, n \in \bN.
\]
Let $e := \dim(V) < d$ and identify
\[
\bM^e \cong \sqcup_{n = 1}^\infty V \otimes M_n \subset \sqcup_{n = 1}^\infty \bC^d \otimes M_n \cong \bM^d. 
\]

Suppose $\im(Q)$ lies in $\cL(\cR, \cS)$ for some Hilbert spaces $\cR$, $\cS$. Let $\{ P_1, \dots, P_e \}$ be a basis for $\im\left(Q \vert_{V}\right) \subset \cL(\cR, \cS)$. Consider the linear operator-valued nc map $P : \bM^e \to \cL (\cR, \cS)_{nc}$ and obtain the corresponding nc operator ball $\bD_P$. It is clear that, under the appropriate identifications, $\fV \subseteq \bD_P \subset \bD_Q$. Also, by the choice of $V \subsetneq \bC^d$ above, we know from Lemma \ref{lemma_description_of_mat-span} that there exists some $n \in \bN$ sufficiently large such that
\[
\operatorname{mat-span}\fV(n) = V \otimes M_n \cong \bC^e \otimes M_n.
\]
Therefore, $\fV$ is matrix-spanning for $\bD_P$ as required, and this completes the proof.
\end{proof}

From now on, we say that the nc operator ball $\bD_P \subset \bM^e$, as given above, is the \emph{minimal nc operator sub-ball} for $\fV \subset \bD_Q$.

\begin{remark}
It is unclear what the structure of $\bD_P$ is with respect to $\fV$ and $\bD_Q$ in general. Therefore, the case $\text{mat-span } \fV = \bM^d$ is of particular interest. When $\bD_Q = \fB_d$, one can choose $\bD_P = \fB_{e}$ for some minimal $e < d$ (see \cite[Theorem 9.10]{SSS18}).

\end{remark}

\subsection{Examples} \label{subsection_examples}

As we saw in Theorem \ref{theorem_cartan_uniqueness_for_subvariety}, we can only say something stronger about the type of nc biholomorphisms between matrix-spanning homogeneous subvarieties for injective nc operator balls. We combine our observations from Sections \ref{subsection_1-injective_nc_operator_balls} and \ref{subsection_matrix-spanning-subvarieties} to provide a couple of interesting examples.

First, we give an example where the minimal nc operator sub-ball corresponding to a homogeneous subvariety $\fV \subset \bD_Q$ is not injective even if $\bD_Q$ is injective.

\begin{example} \label{example_non_1-injective_nc_operator_sub-ball}
Let $\bD_Q \subset \bM^4$ be given by \[
Q(Z) = \begin{bmatrix}
Z_1 & Z_2 \\
Z_4 & Z_3
\end{bmatrix}.
\] Note that $\spn \{ Q_1, Q_2, Q_3, Q_4 \} = B(\bC^2)$, therefore $\bD_Q$ is injective. Consider $\fV = V(\langle Z_4 \rangle)$. It is then easy to check that the minimal nc operator sub-ball of $\fV$ is $\bD_P \subset \bM^3$, given by \[
P(W) = \begin{bmatrix}
W_1 & W_2 \\
0 & W_3
\end{bmatrix}.
\]

We saw in Example \ref{example_1-injective_nc_operator_balls} $(3)$ that $\bD_P$ is not injective. Therefore, our argument does not yield a result like Theorem \ref{theorem_cartan_uniqueness_for_subvariety} for the choice $\bD_{Q_2} = \bD_P$ and $\fV_2 = \fV$ in the theorem.
\end{example}

We now extend the above example to show that, in general, without the injectivity condition for nc operator balls, we cannot guarantee that a nc map between subvarieties extends to the ambient nc operator balls, even if the nc map is a linear isomorphism. In particular, the following example serves as a counter-example for Theorem \ref{thm:extend_to_ball} and the second part of Theorem \ref{thm:iso_cont_1-injective} when you drop the assumption of injectivity from $\bD_{Q_2}$.

\begin{example} \label{example_cannot_guarantee_extension_for_non-injective_balls}
Let $\bD_Q$ and $\bD_P$ be as in Example \ref{example_non_1-injective_nc_operator_sub-ball}. We can think of $\bD_P$ as the homogeneous subvariety $V(\langle Z_4 \rangle) \subset \bD_Q$. Consider the map $F = {\bf id}_{\bD_P} : \bD_P \to \bD_P$. $F$ is clearly a linear isomorphism between two homogeneous nc subvarieties.

Suppose $F$ extends to a nc holomorphic map $\widetilde{F} : \bD_Q \to \bD_P$. Recall that $\Delta \widetilde{F}(0,0)$ is a linear nc map. Now, a result of Rudin (see \cite[Theorem 8.1.2]{Rud08}) says that
\[
\Delta \widetilde{F}(0,0) (\bD_Q(n)) \subseteq \bD_P(n), \, \forall \, n \in \bN.
\]
Combining the above two facts, we get that $\Delta \widetilde{F}(0,0)$ is a linear nc map that maps $\bD_Q$ into $\bD_P$. Recall from (\ref{eq:identify_DQ}) that $\bD_Q$ and $\bD_P$ can be identified with the unit ball of their respective operator spaces $B(\bC^2)$ and $UT_2$. These identifications when conjugated appropriately with $\Delta \widetilde{F}(0,0)$ give us a completely contractive map $\Pi : B(\bC^2) \to UT_2$.

Lastly, note that since $\widetilde{F} \vert_{\bD_P} = {\bf id}_{\bD_P}$, it follows from Lemma \ref{lemma_identity_on_subvariety} that $\Delta \widetilde{F}(0,0) \vert_{\bD_P} = {\bf id}_{\bD_P}$ as well. Thus, $\Pi$ is also a completely contractive projection, which contradicts the fact that $UT_2$ is not injective. Therefore, $F$ cannot be extended to a nc map $\widetilde{F} : \bD_Q \to \bD_P$.
\end{example}

\begin{remark} Note that the map $F : \bD_P \to \bD_P$ in Example \ref{example_cannot_guarantee_extension_for_non-injective_balls} trivially extends to a linear isomorphism on the minimal nc operator sub-balls. Let us see how not having a matrix-spanning subvariety $\fV \subseteq \bD_Q$ affects our discussion.
	
For $i = 1,2$, let $\fV_i \subset \bD_{Q_i}$ be homogeneous subvarieties of injective nc operator balls with minimal nc operator sub-balls $\bD_{P_i} \subsetneq \bD_{Q_i}$ that are not injective. Suppose $G : \fV_1 \to \fV_2$ is a nc biholomorphism with bijective inverse $F$. Then, by Theorem \ref{thm:extend_to_ball}, we know that there are nc maps $\widetilde{G} : \bD_{Q_1} \to \bD_{Q_2}$ and $\widetilde{F} : \bD_{Q_2} \to \bD_{Q_1}$ such that
\[
\widetilde{G}\vert_{\fV_1} = G \text{ and } \widetilde{F} \vert_{\fV_2} = F.
\]
In particular,
\[
\widetilde{F} \circ \widetilde{G} \vert_{\fV_1} = {\bf id}_{\fV_1} \text{ and } \widetilde{G} \circ \widetilde{F} \vert_{\fV_2} = {\bf id}_{\fV_2}.
\]
By Corollary \ref{cor:fix}, it follows that
\[
\widetilde{F} \circ \widetilde{G}\vert_{\bD_{P_1}} = {\bf id}_{\bD_{P_1}} \text{ and } \widetilde{G} \circ \widetilde{F} \vert_{\bD_{P_2}} = {\bf id}_{\bD_{P_2}}.
\]
However, it is not at all clear if $\widetilde{G}(\bD_{P_1}) = \bD_{P_2}$ or $\widetilde{F}(\bD_{P_2}) = \bD_{P_1}$. If we could guarantee this, then we could proceed by replacing $\bD_{Q_i}$ with $\bD_{P_i}$. All we can say, though, is that $\bD_{P_1}$ is nc biholomorphic to its image via $\widetilde{G}$, and that $\bD_{P_2}$ is nc biholomorphic to its image via $\widetilde{F}$, but we do not know what these images are in general.
\end{remark}

We hope that this discussion provides some insight into our use of matrix-spanning homogeneous subvarieties of injective nc operator balls for the sake of the isomorphism problem. Let us now turn our attention to the main classification result of this paper.

\subsection{The main classification result} \label{subsection_the_main_classification_result} We start by showing that if there is a nc biholomorphism between two nc operator balls $\bD_{Q_1}, \bD_{Q_2}$ that maps a subvariety $\fV_1 \subseteq \bD_{Q_1}$ onto $\fV_2 \subseteq \bD_{Q_2}$, then we can find a linear isomorphism between the balls that maps $\fV_1$ onto $\fV_2$.

Helton, Klep and McCullough showed that if two circular nc domains containing $0$ are nc biholomorphic, then there is linear isomorphism between them \cite[Theorem 5.2]{HKM12}. 
However, we cannot use their theorem directly because we need to preserve the subvarieties. We therefore have an intermediate result below that allows us to close this gap. Recall that a nc set $\fX \subset \bM^d$ is said to be \emph{circular} if $\fX$ is invariant under rotations:
\[
0 \leq t < 2 \pi \Rightarrow e^{it} \fX \subseteq \fX.
\]

\begin{theorem} \label{theorem_nc_biholomorphism_that_is_linear}
For $i = 1,2$, let $\Omega_i \subset \bM^{d_i}$ be circular nc domains containing $0$, and let $\fX_i \subseteq \Omega_i$ be closed circular nc subsets containing $0$. Let $G : \Omega_1 \to \Omega_2$ be a nc biholomorphism such that $G(\fX_1) = \fX_2$. Then, there exists a linear isomorphism $L : \Omega_1 \to \Omega_2$ such that $L(\fX_1) = \fX_2$.
\end{theorem}

\begin{proof}
Let $G$ be as in the hypothesis. 
If $G(0) = 0$, then Helton, Klep and McCullough's rigidity theorem \cite[Theorem 4.4]{HKM11a} says that $G$ is a linear isomorphism. 
It remains to show that if $G(0) \neq 0$, then one can find another biholomorphism $\widetilde{G} : \Omega_1 \to \Omega_2$ such that $\widetilde{G}(\fX_1) = \fX_2$ and $\widetilde{G}(0) = 0$. 

For $i = 1,2$, let $D_i = \Omega_i(1)$ and let $\Gamma_i$ be the subgroup of $\Aut(D_i)$ consisting of all biholomorphic automorphisms that extend to a nc automorphism of $\Omega_i$ which fixes $\fX_i$. 
Then $\Gamma_i$ is a subgroup of $\Aut(D_i)$ that contains all rotations (for $0 \leq t < 2 \pi$)
\[
X \mapsto e^{it}X, \, \forall \, X \in \Omega_i
\]
and is closed in the compact open topology (for closedness, use \cite[Proposition 5.11]{HKM12}). By a remarkable result of Braun, Kaup and Upmeier \cite[Theorem 1.2]{BKU78}, 
there exist linear subspaces $V_i$ such that $\Gamma_i(0)$ (the orbit of $0$) is equal to $V_i \cap D_i$ ($i=1,2$). 
Moreover, 
\begin{equation}\label{eq:Gamma}
\Gamma_i(0) = \{z \in D_i : \Gamma_i(z) \textrm{ is a closed submanifold} \}.
\end{equation}

Since $G$ is a nc biholomorphism of $\Omega_1$ onto $\Omega_2$ that maps $\fX_1$ onto $\fX_2$, conjugation with $G$ implements an isomorphism of $\Gamma_1$ and $\Gamma_2$ via \begin{equation} \label{equation_gamma_conjugation}
\gamma \mapsto \widetilde{\gamma} := G \circ \gamma \circ G^{-1}.
\end{equation} 
Equations (\ref{eq:Gamma}) and (\ref{equation_gamma_conjugation}) together imply that $G$ maps $\Gamma_1(0)$ onto $\Gamma_2(0)$. 
Indeed, since
\[
\widetilde{\gamma} \circ G = G \circ \gamma,
\] we have that $G$ maps orbits to orbits, and since $G$ maps closed submanifolds to closed submanifolds we must have
\[
G(\Gamma_1(0)) = \Gamma_2(0).
\]
It follows that $G(0) = \gamma(0)$ for some $\gamma \in \Gamma_2$. 
Thus,
\[
\widetilde{G} = \gamma^{-1} \circ G
\]
is the nc biholomorphism that we set to find in the first paragraph.
\end{proof}

We now state our main classification result for matrix-spanning homogeneous subvarieties of injective nc operator balls. 
Note that there is no uniform continuity assumption. 

\begin{theorem} \label{theorem_isomorphism_A(D_Q)_with_nc_biholomorphism_of_domains}

For $i = 1,2$, let $\fV_i \subseteq \bD_{Q_i}$ be matrix-spanning homogeneous subvarieties of some injective nc operator balls. Then, the following are equivalent:

\begin{enumerate}

\item There exists a completely isometric isomorphism $\alpha : A(\fV_1) \rightarrow A(\fV_2)$.

\item There exists a nc biholomorphism $G : \bD_{Q_1} \rightarrow \bD_{Q_2}$ such that $G({\fV_1}) = \fV_2$.

\item There exists a linear isomorphism $L : \bD_{Q_1} \rightarrow \bD_{Q_2}$ such that $L({\fV_1}) = \fV_2$.

\end{enumerate}

\end{theorem}

\begin{proof}

The implication $(1) \Rightarrow (2)$ follows from Theorems \ref{thm:iso_cont_1-injective} and \ref{theorem_cartan_uniqueness_for_subvariety}, and $(2) \Rightarrow (3)$ follows from \ref{theorem_nc_biholomorphism_that_is_linear}. 
Finally, to see that $(3)$ implies $(1)$, recall that a linear map between finite dimensional operator spaces is completely bounded; it follows that a linear map between nc operator balls is uniformly continuous and so Theorem \ref{thm:iso_cont_1-injective} may be applied. \qedhere

\end{proof}

\begin{example}
The requirement that there exists a linear isomorphism between $\bD_{Q_1}$ and $\bD_{Q_2}$ mapping $\fV_1$ onto $\fV_2$ is very rigid. 
This means that the two operator spaces corresponding to the nc operator balls (see the discussion surrounding \eqref{eq:identify_DQ}) are completely isometrically isomorphic, and in particular isometric. 
Thus, for example, no matrix spanning homogeneous subvarieties $\fV_1 \subseteq \fB_d$ and $\fV_2 \subseteq \fD_d$ can be found so that $A(\fV_1) \cong A(\fV_2)$.
More surprisingly, this non isomorphism persists if $\fD_d$ is replaced by the nc operator ball of an injective Hilbertian operator space that is not completely isometric to the row Hilbert operator space, such as the column Hilbert space. 
\end{example}

\begin{example}
The linear isomorphism requirement is stringent even if the ambient balls $\bD_{Q_1}$ and $\bD_{Q_2}$ are the same, because the linear isometries of a general Banach space might be few. 
For instance, consider the subvarieties \begin{align*}
\fV_1 &= \{X \in \fD_2 : X_1 X_2 - X_2 X_1 = X_1 X_2 = 0\} \\
\fV_2 &= \{X \in \fD_2 : X_1 X_2 - X_2 X_1 = (X_1 + X_2)(X_1 - X_2) = 0\}. 
\end{align*} Even though the algebraic varieties in $\bM^2$ are obtained one from the other by a simple linear change of coordinates, the varieties inside the bidisk $\fD_2$ are not an image one of the other under a linear map preserving $\fD_2$ (check this, for instance, at the first level $\bD^2 = \fD_2(1)$).

However, if we replace $\fV_1, \fV_2$ above with the analogous subvarieties of the nc unit row ball $\fB_2$, then we get varieties that are unitarily related via
\[
U(X_1, X_2) = \left( \frac{X_1 + X_2}{\sqrt{2}}, \frac{X_1 - X_2}{\sqrt{2}} \right), \, \forall (X_1, X_2) \in \fB_2,
\]
and the corresponding nc function algebras are completely isometrically isomorphic. 
\end{example}

Let us now mention some immediate corollaries of Example \ref{example_subvarieties_with_full_matrix_span} and Theorem \ref{theorem_isomorphism_A(D_Q)_with_nc_biholomorphism_of_domains}.

\begin{corollary} \label{corollary_isomorphism_problem_commutative_case}

Let $\fC \bD_{Q_i} \subseteq \bD_{Q_i} \subset \bM^{d_i}$ be two commutative free subvarieties of injective nc operator balls. Then, $A(\fC \bD_{Q_1})$ and $A(\fC \bD_{Q_2})$ are completely isometrically isomorphic if and only if there is a linear isomorphism $L : \bD_{Q_1} \to \bD_{Q_2}$ such that $L(\fC \bD_{Q_1}) = \fC \bD_{Q_2}$.

\end{corollary}

\begin{corollary} \label{corollary_isomorphism_problem_q-commutative_case}

Let $\fC^{q_i} \bD_{Q_i} \subseteq \bD_{Q_i} \subset \bM^2$ be two $q_i$-commuting free subvarieties of injective nc operator balls. Then, $A(\fC^{q_1} \bD_{Q_1})$ and $A(\fC^{q_2} \bD_{Q_2})$ are completely isometrically isomorphic if and only if there is a linear isomorphism $L : \bD_{Q_1} \to \bD_{Q_2}$ such that $L(\fC^{q_1} \bD_{Q_1}) = \fC^{q_2} \bD_{Q_2}$.

\end{corollary}

Lastly, we highlight an interesting application of Theorem \ref{theorem_nc_biholomorphism_that_is_linear} to the case when $\fV_i = \bD_{Q_i}$, which our notations allow by letting $\fV_i = V(\langle 0 \rangle)$ and using Theorem \ref{thm_A(V)_is_quotient_of_A(D_d)} to identify \[
A(\fV_i) \cong A(\bD_{Q_i})/\langle 0 \rangle = A(\bD_{Q_i}).
\]
Note that this result is valid for all nc operator balls $\bD_Q$ and not just the injective ones.

\begin{theorem} \label{theorem_isomorphism_problem_for_V=D_Q}

Let $\bD_{Q_i}$ be two nc operator balls. Then, $A(\bD_{Q_1}) $ and $ A(\bD_{Q_2})$ are completely isometrically isomorphic if and only if $\bD_{Q_1}$ and $\bD_{Q_2}$ are linearly isomorphic.

\end{theorem}

\begin{proof}
Suppose $A(\bD_{Q_1}) \cong A(\bD_{Q_2})$ completely isometrically. By Theorem \ref{thm:iso_cont}, we get that there is a uniformly continuous nc biholomorphism $G : \bD_{Q_1} \to \bD_{Q_2}$ with a uniformly continuous inverse. By Theorem \ref{theorem_nc_biholomorphism_that_is_linear}, it then follows that there is a linear isomorphism $L : \bD_{Q_1} \to \bD_{Q_2}$, as required.

The converse follows from Theorem \ref{thm:iso_cont} and the fact that a linear isomorphism is uniformly continuous with a uniformly continuous bijective inverse.
\end{proof}


\section{\textbf{Concluding Remarks}} \label{section_concluding}

Let us reinterpret some of the above work and try to gain some new insights. 
Suppose $E$ is a finite dimensional space of dimension $d$. 
Let $\cE = E^*$ be the operator space dual of $E$, and let $\{Q_1, \ldots, Q_d\}$ be a basis for $\cE$. 
Define the linear operator-valued polynomial $Q(Z) = \sum_{j = 1}^d Q_j Z_j$.
Then we can define the nc operator ball $\bD_Q$ given by \eqref{eq:operator_ball} and we can identify $\bD_Q(n)$ with the open unit ball of $M_n(\cE)$ as in \eqref{eq:identify_DQ} for each $n \in \bN$. 
We constructed a unital operator algebra $A(\bD_Q)$ of uniformly continuous nc functions on $\bD_Q$. 
Let us see how this operator algebra and the theorems we proved relate to $E$. 

\subsection{$A(\bD_Q)$ is the universal unital operator algebra over $E$}\label{subsec:OA}

We claim that $A(\bD_Q)$ can be naturally identified with the \emph{free unital operator algebra} $OA_u(E)$ generated by $E$ (see \cite[pp. 112--113]{PisierBook}). 
To see this, we first identify
\[
E \cong \spn\{Z_1, \ldots, Z_d\} \subseteq A(\bD_Q).
\]
Now $E = E^{**} = \cE^*$, and thus $\spn\{Z_1, \ldots, Z_d\}$ can be identified as the space of linear functionals on $\cE \cong M^d_1 = \bC^d$. We want to know if this identification is completely isometric. The norm of a matrix valued linear polynomial $F \in \spn\{Z_1, \ldots, Z_d\} \otimes M_k \subset M_k(A(\bD_Q))$ is given by
\[
\sup\{ \|F(X)\| : n \in \bN, X \in \ol{\bD_Q}(n)\},
\]
while the norm of $\Phi \in M_k(\cE^*)$ is, by definition,
\[
\sup\{ \|\Phi(x)\| : n \in \bN,  x \in M_n(\cE), \|x\|_n \leq 1\}. 
\]
It is clear that both these suprema are equal. Therefore, $\spn\{Z_1, \ldots, Z_d\} \cong E$ completely isometrically. 
Under this identification, $\ol{\bD_Q}$ is the space of completely contractive maps from $E$ into $M_n$ for any $n \in \bN$.

After making some additional identifications, $OA_u(E)$ is defined to be the closure of the free nc polynomials $\bC \langle Z \rangle$ with respect to the norm
\begin{align} \label{equation_oaue_norm}
&\left\|\sum_\alpha c_\alpha Z^\alpha\right\|_{OA_u(E)} = \\
&\sup \left\{\left\|c_0 I_{H_v} + \sum_{k=1}^\infty \sum_{|\alpha|=k} c_\alpha v(Z_{\alpha_1}) v(Z_{\alpha_2}) \cdots v(Z_{\alpha_k})\right\| : v : E \to B(H_v), \, \|v\|_{cb} \leq 1 \right\}. \label{equation_oaue_norm_2}
\end{align}
Here, we are using the fact that the tensor algebra (given by a non-closed sum)
\[
T(E) = \bC \oplus E \oplus E \otimes E \oplus \cdots
\]
can be identified with $\bC\langle Z \rangle$, and then we expand all tensors in terms of the basis $\{Z_1, \ldots, Z_d\}$. Matrix norms on $M_n(OA_u(E))$ are defined similarly for each $n \in \bN$.
On the other hand, $A(\bD_Q)$ is the closure of $\bC\langle Z \rangle$ with respect to the norm
\begin{equation} \label{equation_adq_norm}
\left\|\sum_\alpha c_\alpha Z^\alpha\right\|_{A(\bD_Q)} = \sup \left \{\left\|\sum_\alpha c_\alpha X^\alpha \right\| : X \in \ol{\bD_Q} \right \} .
\end{equation}
(\ref{equation_oaue_norm_2}) and the RHS of (\ref{equation_adq_norm}) are almost precisely the same with the one difference that in (\ref{equation_adq_norm}) we restrict to completely contractive maps $v : E \to M_n$ for some $n \in \bN$. 
However, arguing as in the first paragraph of the proof of Lemma \ref{lem:norms}, one can see that these two quantities are equal. 
It is also clear that the above reasoning persists for the corresponding matrix norms as well and hence $A(\bD_Q) \cong OA_u(E)$ completely isometrically. With this concrete identification, we end this discussion with two observations for the reader to ponder over.

\subsection{The meaning of isomorphisms and linear isomorphisms}\label{subsec:linear}

A moment of clarity reveals that the existence of a linear isomorphism $L : \bD_{Q_1} \to \bD_{Q_2}$ is equivalent to the existence of a completely isometric (linear) isomorphism between the corresponding operator spaces $\cE_1$ and $\cE_2$. 
This, in turn, is equivalent to the existence of a completely isometric isomorphism $E_1 \to E_2$. 
Thus, Theorem \ref{theorem_isomorphism_problem_for_V=D_Q} says that two finite dimensional operator spaces $E_1$ and $E_2$ are completely isometrically isomorphic if and only if their corresponding free unital operator algebras $OA_u(E_1)$ and $OA_u(E_2)$ are completely isometrically isomorphic. 
The forward implication is obvious --- if one insists, then it follows from the universal property of $OA_u(E)$ --- but the backward implication is not (it would have been trivial had we assumed that $OA_u(E_1)$ and $OA_u(E_2)$ are completely isometrically isomorphic {\em via an isomorphism that sends $E_1$ to $E_2$}, but getting such a special isomorphism is the nontrivial punchline). 

\subsection{Subvarieties as determining sets}\label{subsec:subvars}

Theorem \ref{theorem_isomorphism_A(D_Q)_with_nc_biholomorphism_of_domains} says that a matrix-spanning homogeneous subvariety $\fV \subseteq \bD_Q$ of an injective nc operator ball is a complete invariant for $A(\fV)$, and more so that the nc operator ball $\bD_Q$ is an invariant (clearly, not a complete one). Even more surprisingly, the theorem says that $\fV$ is {\em determining} for the operator space $E$. 

To be explicit, consider Corollary \ref{corollary_isomorphism_problem_commutative_case}. 
Let us say that a completely contractive map is {\em commuting} if its range generates a commutative algebra. 
Now, Theorem \ref{theorem_isomorphism_problem_for_V=D_Q} implies that all the completely contractive representations of an injective operator space on a finite dimensional space form a complete invariant of the operator space --- sure, that's obvious (well, if we assume $\dim E < \infty$). 
Corollary \ref{corollary_isomorphism_problem_commutative_case} then implies that already the commuting completely contractive representations of some injective $E$ on a finite dimensional space determine $E$, in the sense that if the commuting representations of $E_1$ and $E_2$ are nc biholomorphic, then $E_1$ and $E_2$ must be completely isometrically isomorphic. 

\subsection*{Acknowledgements} The authors are grateful to Eli Shamovich who read an earlier version of this paper and found a critical mistake in the previous formulation of Theorem 6.5. The authors are also grateful to an anonymous referee for suggesting several corrections and improvements to the paper.


\bibliographystyle{amsplain}

\begin{thebibliography}{99}

\bibitem{Abate} M. Abate. 
\emph{Iteration theory of holomorphic maps on taut manifolds}. 
Research and Lecture Notes in Mathematics. Complex Analysis and Geometry. Mediterranean Press, Rende, 1989.

%
\bibitem{AMBook}
J. Agler and J.E. McCarthy, 
\newblock {\em {Pick Interpolation and Hilbert Function Spaces}},
\newblock Graduate Studies in Mathematics, 44, Providence,
\newblock RI: American Mathematical Society, (2002).

\bibitem{AM15a} J. Agler and J.E. McCarthy, 
\emph{Global holomorphic functions in several non-commuting variables}, 
Canadian J. Math. \textbf{67} (2015), 241--285. 

\bibitem{AM15b} J. Agler and J.E. McCarthy, 
\emph{Pick interpolation for free holomorphic functions}, 
Amer. J. Math. \textbf{137} (2015), 1685--1701.

\bibitem{AM16} J. Agler and J.E. McCarthy, 
\emph{Aspects of non-commutative function theory},
Concrete Operators 3.1 (2016), 15--24.



\bibitem{AriLat12} A. Arias and F. Latremoliere, 
\emph{Classification of noncommutative domain algebras}, 
C. R. Math. Acad. Sci. Paris \textbf{350} (2012), 609--611.

\bibitem{AriPop00} A. Arias and G. Popescu, 
\emph{Noncommutative interpolation and Poisson transforms}, 
Israel J. Math. \textbf{115} (2000), 205--234.

\bibitem{Arv69} W.B. Arveson, 
\emph{Subalgebras of C*-algebras}, 
Acta Math. \textbf{123} (1969), 141--224.

\bibitem{Arv98} W.B. Arveson,
\emph{Subalgebras of C*-algebras III: Multivariable operator theory},
Acta Math. \textbf{181} (1998).

\bibitem{Arv11} W.B. Arveson, 
\emph{The noncommutative Choquet boundary II: hyperrigidity}, 
Israel J. Math. \textbf{184} (2011): 349--385.

\bibitem{BMV16} J.A. Ball, G. Marx and V. Vinnikov, 
\emph{Noncommutative reproducing kernel Hilbert spaces}, 
J. Funct. Anal. \textbf{271} (2016), 1844--1920.

\bibitem{BMV18} J.A. Ball, G. Marx and V. Vinnikov, 
\emph{Interpolation and transfer-function realization for the noncommutative Schur-Agler class}, 
In Operator theory in different settings and related applications, Birkh\"{a}user, Cham, (2018), 23--116.

\bibitem{BS+} S. Belinschi and E. Shamovich, 
\emph{Iteration theory of noncommutative maps}, arXiv:2310.03549, (2023). 

\bibitem{BLM04} D.P. Blecher and C.L. Merdy, 
\emph{Operator algebras and their modules: an operator space approach}, 
Oxford University Press, (2004).

\bibitem{BKU78} R. Braun, W. Kaup and H. Upmeier,
\emph{On the automorphisms of circular and Reinhardt domains in complex Banach spaces},
Manuscripta Math \textbf{25} 97–133 (1978).

\bibitem{DavPitts98a} K.R. Davidson and D.R. Pitts, 
\emph{Nevanlinna-Pick interpolation for non-commutative analytic Toeplitz algebras}, 
Integral Equations Operator Theory \textbf{31} (1998), 321--337.

\bibitem{DavPitts98b} K.R. Davidson and D.R. Pitts, 
\emph{The algebraic structure of non-commutative analytic Toeplitz algebras}, 
Math. Ann. \textbf{311} (1998), 275--303.

\bibitem{DavPitts99} K.R. Davidson and D.R. Pitts, 
\emph{Invariant subspaces and hyper-reflexivity for free semigroup algebras}, 
Proc. Lond. Math. Soc. \textbf{78} (1999), 401--430.

\bibitem{DRS11} K.R. Davidson, C. Ramsey and O.M. Shalit,
\emph{The isomorphism problem for some universal operator algebras},
Adv. Math. \textbf{228} (2011), 167--218.

\bibitem{DRS15} K.R. Davidson, C. Ramsey and O.M. Shalit, 
\emph{Operator algebras for analytic varieties}, 
Trans. Amer. Math. Soc. \textbf{367} (2015), 1121--1150. 

\bibitem{Har12} M. Hartz, 
\emph{Topological isomorphisms for some universal operator algebras}, 
J. Funct. Anal. \textbf{263} (2012).

\bibitem{HRS22} M. Hartz, S. Richter and O.M. Shalit, 
\emph{von Neumann's inequality for row contractive matrix tuples}, 
Math. Z \textbf{301} (2022), 3877--3894.

\bibitem{HarSha23} M. Hartz and O.M. Shalit,  
\emph{Tensor algebras of subproduct systems and noncommutative function theory}, 
arxiv preprint, (2023).

\bibitem{HKM11a} J.W. Helton, I. Klep and S.A. McCullough, 
\emph{Proper analytic free maps},
J. Funct. Anal. \textbf{260} (2011), 1476--1490.

\bibitem{HKM11b} J.W. Helton, I. Klep and S.A. McCullough, 
\emph{Analytic mappings between noncommutative pencil balls}, 
J. Math. Anal. Appl. \textbf{376} (2011), 407--228.

\bibitem{HKM12} J.W. Helton, I. Klep and S.A. McCullough, 
\emph{Free analysis, convexity and LMI domains}, 
Mathematical Methods in Systems, Optimization, and Control: Festschrift in Honor of J. William Helton. Basel: Springer Basel, (2012), 195--219.

%
%
\bibitem{KakSha19} E.T.A. Kakariadis and O.M. Shalit, 
\emph{Operator algebras of monomial ideals in noncommuting variables}, 
J. Math. Anal. Appl. \textbf{472} (2019), 738--813.

\bibitem{KVV14} D. Kaliuzhnyi-Verbovetskyi and V. Vinnikov,
\emph{Foundations of Free Noncommutative Function Theory}, American Mathematical Soc., \textbf{199} (2014).

\bibitem{Kiri} R. Kiri, 
\emph{Subhomogeneous Operator Systems and Classification of Operator Systems Generated by $\Lambda $-Commuting Unitaries}, 
arXiv preprint arXiv:2302.04767 (2023).


\bibitem{MS98} P.S. Muhly and B. Solel,
\textit{Tensor algebras over C*-correspondences: representations, dilations and C*-envelopes},
J. Funct.\ Anal.\ \textbf{158} (1998), 389--457.


\bibitem{MS04} P.S. Muhly and B. Solel,
\emph{Hardy algebras, W*-correspondences and interpolation theory}, 
Math. Ann. \textbf{330} (2004), 353--415.

\bibitem{MS05} P.S. Muhly and B. Solel,  
\emph{Hardy algebras associated with W*-correspondences (point evaluation and Schur class functions)}, 
Operator Theory, systems theory and scattering theory: multidimensional generalizations. Basel: Birkh\"{a}user Basel, (2005), 221--241.

\bibitem{MS11} P.S. Muhly and Baruch Solel, 
\emph{Representations of Hardy algebras: absolute continuity, intertwiners, and superharmonic operators}, 
Integral Equations Operator Theory \textbf{70} (2011), 151--203.

\bibitem{PaulsenBook} V. Paulsen, 
\emph{Completely Bounded Maps and Operator Algebras}, 
Cambridge University Press, (2002).

\bibitem{PisierBook} G. Pisier,
\emph{Introduction to Operator Space Theory},
Cambridge University Press, (2003).

\bibitem{PV18} M. Popa and V. Vinnikov,
\emph{$H^2$ spaces of non-commutative functions},
Complex Analysis and Operator Theory \textbf{12} (2018), 945--967.

\bibitem{Pop95} G. Popescu, 
\emph{Functional calculus for noncommuting operators}, 
Michigan Math. J. \textbf{42} (1995), 345--352.

\bibitem{Pop96} G. Popescu, 
\emph{Non-commutative disk algebras and their representations}, 
Proc. Amer. Math. Soc. \textbf{124}:7 (1996), 2137--2148.

\bibitem{Pop06} G. Popescu, 
\emph{Operator theory on noncommutative varieties}, 
Indiana Univ. Math. J. \textbf{55} (2006).

\bibitem{Pop11} G. Popescu, 
\emph{Free biholomorphic classification of noncommutative domains}, 
Int. Math. Res. Not. IMRN (2011), 784--850.

\bibitem{Ruan89} Z.J Ruan,
\emph{Injectivity of operator spaces},
Transactions of the Amer. Math. Soc. \text{315}(1) (1989), 89--104

\bibitem{Rud08} W. Rudin,
\emph{Function theory in the unit ball of $\bC^n$},
Springer Science \& Business Media (2008).

\bibitem{Sha15} O.M. Shalit,
\emph{Operator Theory and Function Theory in Drury–Arveson Space and Its Quotients},
In: Alpay, D. (eds) Operator Theory. Springer, Basel (2015).

\bibitem{Smith00} R. Smith,
\emph{Finite dimensional injective operator spaces},
Proceedings of the Amer. Math. Soc. \textbf{128}(11) (2000), 3461--3462.

\bibitem{SSS18} G. Salomon, O.M. Shalit and E. Shamovich, 
\emph{Algebras of bounded noncommutative analytic functions on subvarieties of the noncommutative unit ball}, 
Trans. Amer. Math. Soc. \textbf{370} (2018), 8639--8690.

\bibitem{SSS20} G. Salomon, O.M. Shalit and Eli Shamovich, 
\emph{Algebras of noncommutative functions on subvarieties of the noncommutative ball: the bounded and completely bounded isomorphism problem}, 
J. Funct. Anal. \textbf{278} (2020), 108427.


\bibitem{ShaSol09} O.M. Shalit and B. Solel,
\textit{Subproduct systems},
Doc. Math. \textbf{14} (2009), 801--868.

\bibitem{Sham18} E. Shamovich, 
\emph{On fixed points of self maps of the free ball}, 
J. Funct. Anal. \textbf{275} (2018), 422--441.

\end{thebibliography}

\end{document}